\documentclass[11pt,a4paper]{amsart}

	\usepackage[USenglish]{babel}
	\usepackage[latin1]{inputenc}    			
	\usepackage[T1]{fontenc}							

	\usepackage[bbgreekl]{mathbbol}

	\usepackage{amsfonts,amsmath,amssymb,amsthm}
	\usepackage{times}
	\usepackage{color}
	\usepackage{amscd}
	\usepackage{framed} 									
	\usepackage{upgreek}

 	\usepackage{multirow,multicol}
	\usepackage{calc}
	\usepackage{xspace}
	\usepackage{eucal}
	\usepackage{dsfont}\DeclareSymbolFont{rsfs}{U}{rsfs}{m}{n}
	\DeclareSymbolFontAlphabet{\mathrsfs}{rsfs}
	\usepackage{color}
	\usepackage{enumitem} 
	\usepackage{cite}
	\usepackage[arrow, matrix, curve]{xy}	
	
  \newtheoremstyle{style}   
		{}												    
		{}      			  	            
	  {\itshape}                    
	 	{}                            
		{\normalfont\bfseries}	 	 
		{\normalfont\bfseries .} {.5em}		
		{}
	\theoremstyle{style}
	\newtheorem{theorem}{Theorem}[section]
	\newtheorem{lemma}[theorem]{Lemma}
	\newtheorem{corollary}[theorem]{Corollary}
	\newtheorem{proposition}[theorem]{Proposition}

	\newtheorem{example}[theorem]{Example}
	
	\newtheorem{introtheorem}{Theorem}

	\newtheorem{question}[introtheorem]{Question}	

	\theoremstyle{remark}
	\newtheorem*{remark}{\textbf{Remark}}
	\newtheorem*{hypothesis}{\textbf{Hypothesis}}

	\theoremstyle{definition}
	\newtheorem{definition}[theorem]{Definition}
	
	\theoremstyle{definition}

	\theoremstyle{definition}
	
	\numberwithin{equation}{section}
	
	\allowdisplaybreaks[4]							
  \title{On the difference of spectral projections}
	\date{\today}

  \author[C. Uebersohn]{Christoph Uebersohn}

	\address{\begin{align*}
						&\text{C.~Uebersohn,} \\ &\text{FB 08 - Institut f\"{u}r Mathematik, Johannes Gutenberg-Universit\"{a}t Mainz,} \\
						&\text{Staudinger Weg 9, D-55099 Mainz, Germany}
					\end{align*}}
					\email{uebersoc@uni-mainz.de}

\begin{document}
	\begin{abstract}
		For a semibounded self-adjoint operator $ T $ and a compact self-adjoint operator $ S $ acting on 
		a complex separable Hilbert space of infinite dimension, we study the difference
		$ D(\lambda) := E_{(-\infty, \lambda)}(T+S) - E_{(-\infty, \lambda)}(T), \, \lambda \in \mathds{R} $, 
		of the spectral projections associated with the open interval $ (-\infty, \lambda) $. 
		
		In the case when $ S $ is of rank one, we show that $ D(\lambda) $ is unitarily equivalent to a block diagonal operator 
		$ \Gamma_{\lambda} \oplus 0 $, where $ \Gamma_{\lambda} $ is a bounded self-adjoint Hankel operator, 
		for all $ \lambda \in \mathds{R} $ except for at most countably many $ \lambda $.
	
		If, more generally, $ S $ is compact, then we obtain that $ D(\lambda) $ is unitarily equivalent to an essentially Hankel operator (in the sense of Mart{\'{\i}}nez-Avenda{\~n}o) 
		on $ \ell^{2}(\mathds{N}_{0}) $ for all $ \lambda \in \mathds{R} $ except for at most countably many $ \lambda $.
	\end{abstract}
\maketitle

\section{Introduction and main results} \label{Abschnitt mit Hauptergebnissen}
	When a self-adjoint operator $ T $ is perturbed by a bounded self-adjoint operator $ S $, it is important to investigate the (spectral) properties of the difference
	\begin{equation*}
		f(T+S) - f(T),
	\end{equation*}
	where $ f $ is a real-valued Borel function on $ \mathds{R} $.  
	It is also of interest to predict the smoothness of the mapping $ S \mapsto f(T+S) - f(T) $ with respect to the smoothness of $ f $.	
	There is a vast amount of literature dedicated to these problems, see, e.\,g., Kre{\u\i}n, Farforovskaja, Peller, Birman, Solomyak, Pushnitski, Yafaev 
	\cite{Krein, Krein_II, Farforovskaja, Peller_I, Peller_II, Birman_Solomyak, Pushnitski_I, Pushnitski_II, Pushnitski_III, Pushnitski_Yafaev}, and the references therein.

	It is well known (see Kre{\u\i}n \cite{Krein}; see also Peller \cite{Peller_II}) that if $ f $ is an infinitely differentiable function with compact support and $ S $ is trace class, 
	then $ f(T+S) - f(T) $ is a trace class operator.
	
	On the other hand, if $ f = \mathds{1}_{(-\infty, \lambda)} $ is the characteristic function of the interval $ (-\infty, \lambda) $ with $ \lambda $ in the essential spectrum 
	of $ T $, then it may occur that
	\begin{equation*}
		f(T+S) - f(T)
	\end{equation*}
	is not compact, see Kre{\u\i}n's example \cite{Krein,Kostrykin}. 
	In the latter example, $ S $ is a rank one operator, and the difference $ \mathds{1}_{(-\infty, \lambda)}(T+S) - \mathds{1}_{(-\infty, \lambda)}(T) $ is 
	a bounded self-adjoint Hankel integral operator on $ L^{2}(0,\infty) $ that can be computed explicitly for all $ 0 < \lambda < 1 $.
	
	Formally, a bounded Hankel integral operator $ \Gamma $ on $ L^{2}(0,\infty) $ is a bounded integral operator 
	such that the kernel function $ k $ of $ \Gamma $ depends only on the sum of the variables:
	\begin{equation*}
		(\Gamma g)(x) = \int_{0}^{\infty} k(x+y) g(y) \mathrm{d}y, \quad g \in L^{2}(0,\infty).
	\end{equation*}
	For an introduction to the theory of Hankel operators, we refer to Peller's book \cite{Peller}.

	Inspired by Kre{\u\i}n's example, we may pose the following question.
	\begin{question}	\label{The main question}
		Let $ \lambda \in \mathds{R} $. Is it true that 
		\begin{equation*}
			D(\lambda) = E_{(-\infty, \lambda)}(T+S) - E_{(-\infty, \lambda)}(T),
		\end{equation*}
		the difference of the spectral projections, is unitarily equivalent to a bounded self-adjoint Hankel operator, provided that $ T $ is semibounded and $ S $ is of rank one?
	\end{question}
	
	Pushnitski \cite{Pushnitski_I, Pushnitski_II, Pushnitski_III, Pushnitski_Yafaev} and Yafaev \cite{Pushnitski_Yafaev} have been studying the spectral properties of 
	the operator $ D(\lambda) $ in connection with scattering theory. 
	If the absolutely continuous spectrum of $ T $ contains an open interval and under some smoothness assumptions, the results of Pushnitski and Yafaev are applicable, 
	cf.\,Section \ref{The section with finite rank perturbation} below. 
	In this case, the essential spectrum of $ D(\lambda) $ is a symmetric interval around zero.
	
	Here and for the rest of this paper, we consider a semibounded self-adjoint operator $ T $ acting on a complex separable Hilbert space $ \mathfrak{H} $ 
	of infinite dimension. 
	We denote the spectrum and the essential spectrum of $ T $ by $ \sigma(T) $ and $ \sigma_{\mathrm{ess}}(T) $, respectively. 
		
	Furthermore, we denote by $ \mathrm{span} \{ x_{i} \in \mathfrak{H} : i \in \mathcal{I} \} $ the linear span generated by the vectors $ x_{i} $, $ i \in \mathcal{I} $, where 
	$ \mathcal{I} $ is some index set. If there exists a vector $ x \in \mathfrak{H} $ such that 
	\begin{equation*}
		\overline{\mathrm{span}} \left\{ E_{\Omega}(T) x : \Omega \in \mathcal{B}(\mathds{R}) \right\}
		:= \overline{\mathrm{span} \left\{ E_{\Omega}(T) x : \Omega \in \mathcal{B}(\mathds{R}) \right\}}
		= \mathfrak{H},
	\end{equation*}
	then $ x $ is called \emph{cyclic} for $ T $. Here $ \mathcal{B}(\mathds{R}) $ denotes the sigma-algebra of Borel sets of $ \mathds{R} $.
	
	The following theorem is the main result of this paper. 
	\begin{introtheorem}	\label{new Main result}
		Let $ T $ and $ S $ be a semibounded self-adjoint operator and a self-adjoint operator of rank one acting on $ \mathfrak{H} $, respectively.
		~Then there exists a number $ k $ in $ \mathds{N} \cup \{ 0 \} $ such that for all $ \lambda $ in $ \mathds{R} $ except for 
		at most countably many $ \lambda $ in $ \sigma_{\mathrm{ess}}(T) $, the operator $ D(\lambda) $ is unitarily equivalent to a block diagonal operator 
		$ \Gamma_{\lambda} \oplus 0 $ on $ L^{2}(0, \infty) \oplus \mathds{C}^{k} $, where $ \Gamma_{\lambda} $ is a bounded self-adjoint Hankel integral operator on $ L^{2}(0, \infty) $.
	\end{introtheorem}
	The theory of bounded self-adjoint Hankel operators has been studied intensively by Rosenblum, Howland, Megretski{\u\i}, Peller, Treil, and others, see 
	\cite{Rosenblum_I, Rosenblum_II, Howland, Megretskii_et_al}. 
	In their 1995 paper \cite{Megretskii_et_al}, Megretski{\u\i}, Peller, and Treil have shown that every bounded self-adjoint Hankel operator can be characterized by three properties 
	concerning the spectrum and the multiplicity in the spectrum, see \cite[Theorem 1]{Megretskii_et_al}.
	
	We present a version of \cite[Theorem 1]{Megretskii_et_al} for differences of two orthogonal projections in Section \ref{Formulierung der Hilfsmittel}, 
	see Theorem \ref{Charakterisierung modifiziert} below.

	Denote by $ \ell^{2}(\mathds{N}_{0}) $ the space of all complex square summable one-sided sequences $ x=(x_{0}, x_{1}, ...) $. 
	A bounded operator $ H $ on $ \ell^{2}(\mathds{N}_{0}) $ is called \textit{essentially Hankel} if $ A^{\ast} H - H A $ is compact, 
	where $ A : (x_{0}, x_{1}, ...) \mapsto (0, x_{0}, x_{1}, ...) $ denotes the forward shift on $ \ell^{2}(\mathds{N}_{0}) $. 
	The set of essentially Hankel operators was	introduced in \cite{Martinez} by Mart{\'{\i}}nez-Avenda{\~n}o. 
	Clearly, every operator of the form 'Hankel plus compact' is essentially Hankel, 
	but the converse is not true (see \cite[Theorem 3.8]{Martinez}).
	
	For compact perturbations $ S $, we will prove the following version of Theorem \ref{new Main result}. \newpage
	\begin{introtheorem}	\label{new Main result II}
		Let $ T $ and $ S $ be a semibounded self-adjoint operator and a compact self-adjoint operator acting on $ \mathfrak{H} $, respectively. Let $ 1/4 > a_{1} > a_{2} > ... > 0 $
		be an arbitrary decreasing null sequence of real numbers. 
		~Then for all $ \lambda $ in $ \mathds{R} $ except for at most countably many $ \lambda $ in $ \sigma_{\mathrm{ess}}(T) $, 
		there exist a bounded self-adjoint Hankel operator $ \Gamma_{\lambda} $ and a compact self-adjoint operator $ K_{\lambda} $ acting on $ \ell^{2}(\mathds{N}_{0}) $  
		with the following properties:
		\begin{enumerate}
			\item $ D(\lambda) $ is unitarily equivalent to $ \Gamma_{\lambda} + K_{\lambda} $.
			\item either $ K_{\lambda} $ is a finite rank operator or $ \nu_{j}(\lambda) / a_{j} \rightarrow 0 $ as $ j \rightarrow \infty $, 
						where $ \nu_{1}(\lambda), \nu_{2}(\lambda), ... $ denote the nonzero eigenvalues of $ K_{\lambda} $ ordered by decreasing modulus (with multiplicity taken into account).
		\end{enumerate}
		In particular, $ \Gamma_{\lambda} + K_{\lambda} $ is essentially Hankel.
	\end{introtheorem}
	Moreover, the operator $ K_{\lambda} $ in Theorem \ref{new Main result II} can always be chosen of finite rank if $ S $ is of finite rank.
	
	In Sections \ref{The second section}--\ref{The section with finite rank perturbation}, the operator $ T $ is supposed to be bounded. 
	The case when $ T $ is semibounded (but not bounded) will be reduced to the bounded case by means of resolvents, 
	see Subsection \ref{Zum halbbeschraenkten Fall} and the remark in Subsection \ref{Subsection with two auxiliary results}.

	In Section \ref{The second section}, we will show that the dimensions of $ \mathrm{Ker} (D(\lambda) \pm I) $ differ by at most $ N \in \mathds{N} $ if 
	$ S $ is of rank $ N $, where $ I $ denotes the identity operator. We write this as
	\begin{equation}	\label{Ungleichung zur Symmetrie des Spektrums}
		\big| \mathrm{dim} ~ \mathrm{Ker} \big( D(\lambda) - I \big) - \mathrm{dim} ~ \mathrm{Ker} \big( D(\lambda) + I \big) \big| \leq \mathrm{rank} ~ S,
		\quad \lambda \in \mathds{R}.
	\end{equation}
	Furthermore, an example is given where equality is attained.
	
	However, there may exist $ \lambda \in \mathds{R} $ such that 
	\begin{equation*}
		\mathrm{dim} ~ \mathrm{Ker} \big( D(\lambda) - I \big) = \infty		\quad		\text{and}		\quad		\mathrm{Ker} \big( D(\lambda) + I \big) = \{ 0 \}
	\end{equation*}
	if $ S $ is a compact operator with infinite dimensional range.

	Section \ref{The third section} provides a list of sufficient conditions so that Question \ref{The main question} has a positive answer, 
	see Proposition \ref{Einige hinreichende Bedingungen}. 
	
	Moreover, if $ S = \langle \cdot, \varphi \rangle \varphi $ is a rank one operator and the vector $ \varphi $ is cyclic for $ T $, 
	then we will show in Theorem \ref{Main theorem II} that the kernel of $ D(\lambda) $ is trivial for all $ \lambda $ in the \linebreak
	interval $ ( \min \sigma_{\mathrm{ess}}(T), \max \sigma_{\mathrm{ess}}(T) ) $ and infinite dimensional for all $ \lambda $ 
	in \linebreak $ \mathds{R} \setminus [ \min \sigma_{\mathrm{ess}}(T), \max \sigma_{\mathrm{ess}}(T) ] $.
	
	In the case when $ \varphi $ is not cyclic for $ T $, Example \ref{via Krein} shows that Question \ref{The main question} may have to be answered negatively. 
	In this situation, we need to consider the block operator representation of $ D(\lambda) $ with respect to the orthogonal subspaces
	$ \overline{\mathrm{span}} \{ T^{j} \varphi : j \in \mathds{N}_{0} \} $ 
	and $ \mathfrak{H} \ominus \overline{\mathrm{span}} \{ T^{j} \varphi : j \in \mathds{N}_{0} \} $ of $ \mathfrak{H} $, 
	see Subsection \ref{Reduktion zum zyklischen Fall}.
	
	In Section \ref{The section with finite rank perturbation}, we will show that the operator $ D(\lambda) $ is non-invertible for all $ \lambda $ in $ \mathds{R} $ 
	except for at most countably many $ \lambda $ in $ \sigma_{\mathrm{ess}}(T) $, see Theorem \ref{Main theorem III}.
	
	Section \ref{Beweis des Hauptresultates} completes the proofs of Theorems \ref{new Main result} and \ref{new Main result II}. 
	In particular, it is shown that $ D(\lambda) $ is unitarily equivalent to a self-adjoint Hankel operator of finite rank for \emph{all} $ \lambda \in \mathds{R} $ 
	if $ T $ has a purely discrete spectrum and $ S $ is a rank one operator 
	(see Proposition \ref{Keine Ausnahmepunkte im Fall von purely discrete spectrum} and p.\,\pageref{Ergaenzung Proposition}).
	
	Some examples, including the almost Mathieu operator, are discussed in Section \ref{Beispiele und Anwendungen} below.
	
	The results of this paper will be part of the author's Ph.D. thesis at Johannes Gutenberg University Mainz.
	
	\section{The main tools} \label{Formulierung der Hilfsmittel}
	
	In this section, we present the main tools for the proofs of Theorems \ref{new Main result} and \ref{new Main result II}. 
	First, we state a lemma which follows immediately from \cite[Theorem 6.1]{Davis}.
	
	\begin{lemma} \label{Satz bei Chandler Davis}
		Let $ \Gamma $ be the difference of two orthogonal projections. \, Then $ \sigma(\Gamma) \subset [-1,1] $. 
		Moreover, the restricted operators $ \left. \Gamma \right|_{\mathfrak{H}_{0}} $ and $ \left. (-\Gamma) \right|_{\mathfrak{H}_{0}} $ are unitarily equivalent, where 
		the closed subspace $ \mathfrak{H}_{0} := \left[ \mathrm{Ker} ~ (\Gamma - I) \oplus \mathrm{Ker} ~ (\Gamma + I) \right]^{\bot} $ of $ \mathfrak{H} $ is reducing for $ \Gamma $.
	\end{lemma}
	
	In \cite{Megretskii_et_al}, Megretski{\u\i}, Peller, and Treil solved the inverse spectral problem for self-adjoint Hankel operators. 
	In our situation, \cite[Theorem 1]{Megretskii_et_al} reads as follows:
	
	\begin{theorem} \label{Charakterisierung modifiziert}
		The difference $ \Gamma $ of two orthogonal projections is unitarily equivalent to a bounded self-adjoint Hankel operator if and only if the following 
		three conditions hold:
		\begin{itemize}
			\item[\emph{(C1)}] either $ \mathrm{Ker} ~ \Gamma = \{ 0 \} $ or $ \mathrm{dim} ~ \mathrm{Ker} ~ \Gamma = \infty $;
			\item[\emph{(C2)}] $ \Gamma $ is non-invertible;
			\item[\emph{(C3)}] $ | \mathrm{dim} ~ \mathrm{Ker} (\Gamma - I) - \mathrm{dim} ~ \mathrm{Ker} (\Gamma + I) | \leq 1 $.
		\end{itemize}
	\end{theorem}
	If $ \mathrm{dim} ~ \mathrm{Ker} (\Gamma - I) = \infty $ or $ \mathrm{dim} ~ \mathrm{Ker} (\Gamma + I) = \infty $, 
	then (C3) has to be understood as $ \mathrm{dim} ~ \mathrm{Ker} (\Gamma - I) = \mathrm{dim} ~ \mathrm{Ker} (\Gamma + I) = \infty $ (cf.\,\cite[p.\,249]{Megretskii_et_al}). 
	\begin{proof}[Proof of Theorem \ref{Charakterisierung modifiziert}]
		Combine Lemma \ref{Satz bei Chandler Davis} and \cite[Theorem 1]{Megretskii_et_al}.
	\end{proof}
	As will be shown in Section \ref{The second section}, the operator $ D(\lambda) $ satisfies condition (C3) for all $ \lambda \in \mathds{R} $ if $ S $ is a rank one operator. 
	Therefore, a sufficient condition for $ D(\lambda) $ to be unitarily equivalent to a bounded self-adjoint Hankel operator is given by:
	
	\emph{the kernel of $ D(\lambda) $ is infinite dimensional.}
	
	In Proposition \ref{Einige hinreichende Bedingungen} below, we present a list of sufficient conditions such that the kernel of $ D(\lambda) $ is infinite dimensional. 
	
	More generally, a self-adjoint \textit{block-Hankel operator of order $ N $} is a block-Hankel matrix $ (a_{j+k})_{j,k \in \mathds{N}_{0}} $, where $ a_{j} $ is an $ N \times N $ 
	matrix for every $ j $, see \cite[p.\,247]{Megretskii_et_al}. We will need the following version of Theorem \ref{Charakterisierung modifiziert}:
	
	\begin{theorem} \label{Charakterisierung modifiziert allgemeinere Version}
		The difference $ \Gamma $ of two orthogonal projections is unitarily equivalent to a bounded self-adjoint block-Hankel operator of order $ N $ if and only if the following 
		three conditions hold:
		\begin{itemize}
			\item[\emph{(C1)}] \quad either $ \mathrm{Ker} ~ \Gamma = \{ 0 \} $ or $ \mathrm{dim} ~ \mathrm{Ker} ~ \Gamma = \infty $;
			\item[\emph{(C2)}] \quad $ \Gamma $ is non-invertible;
			\item[\emph{(C3)}] \hspace{-1ex}$ _{N} $ \, $ | \mathrm{dim} ~ \mathrm{Ker} (\Gamma - I) - \mathrm{dim} ~ \mathrm{Ker} (\Gamma + I) | \leq N $.
		\end{itemize}
	\end{theorem}
	Again, if $ \mathrm{dim} ~ \mathrm{Ker} (\Gamma - I) = \infty $ or $ \mathrm{dim} ~ \mathrm{Ker} (\Gamma + I) = \infty $, 
	then (C3)$ _{N} $ has to be understood as $ \mathrm{dim} ~ \mathrm{Ker} (\Gamma - I) = \mathrm{dim} ~ \mathrm{Ker} (\Gamma + I) = \infty $. 
	\begin{proof}[Proof of Theorem \ref{Charakterisierung modifiziert allgemeinere Version}]
		Combine Lemma \ref{Satz bei Chandler Davis} and \cite[Theorem 2]{Megretskii_et_al}.
	\end{proof}
	
	\section{On the dimension of $ \mathrm{Ker} \big( D(\lambda) \pm I \big) $}	\label{The second section}
	
	In this section, the self-adjoint operator $ T $ is assumed to be bounded. 
	The main purpose of this section is to show that the dimensions of $ \mathrm{Ker} \big( D(\lambda) \pm I \big) $ 
	do not exceed the rank of the perturbation $ S $, see Lemma \ref{Lemma zur Bedingung C3} below.
	
	In particular, condition (C3)$ _{N} $ in Theorem \ref{Charakterisierung modifiziert allgemeinere Version} is fulfilled for all $ \lambda \in \mathds{R} $ if 
	the rank of $ S $ is equal to $ N \in \mathds{N} $.
	
	\begin{lemma}	\label{Lemma zur Bedingung C3}
		Let $ T $ and $ S $ be a bounded self-adjoint operator and a self-adjoint operator of finite rank $ N $ acting on $ \mathfrak{H} $, respectively.
		Then for all $ \lambda $ in $ \mathds{R} $, one has
		\begin{equation*}
			\mathrm{dim} ~ \mathrm{Ker} \big( D(\lambda) \pm I \big) \leq N.
		\end{equation*}
	\end{lemma}
	\begin{proof}
		Let us write $ P_{\lambda} := E_{(-\infty, \lambda)}(T+S) $ and $ Q_{\lambda} := E_{(-\infty, \lambda)}(T) $.
		
		We will only show that $ \mathrm{dim} ~ \mathrm{Ker} ( P_{\lambda} - Q_{\lambda} - I ) \leq N $; the other inequality is proved analogously.
		
		Assume for contradiction that there exists an orthonormal system $ x_{1}, ..., x_{N+1} $ in $ \mathrm{Ker} ( P_{\lambda} - Q_{\lambda} - I ) $. 
		Choose a normed vector $ \tilde{x} $ in 
		\begin{equation*}
			\mathrm{span} \{ x_{1}, ..., x_{N+1} \} \cap (\mathrm{Ran} ~ S)^{\bot} \neq \{ 0 \}.
		\end{equation*}
		Hence $ P_{\lambda} \tilde{x} = \tilde{x} $ and $ Q_{\lambda} \tilde{x} = 0 $ and this implies
		\begin{align*}
			\langle (T+S) \tilde{x}, \tilde{x} \rangle < \lambda \quad \text{and} \quad \langle T \tilde{x}, \tilde{x} \rangle \geq \lambda
		\end{align*}
		so that
		\begin{align*}
			\lambda > \langle (T+S) \tilde{x}, \tilde{x} \rangle = \langle T \tilde{x}, \tilde{x} \rangle \geq \lambda,
		\end{align*}
		which is a contradiction.
	\end{proof}
	
	\begin{remark}
		If we consider an unbounded self-adjoint operator $ T $, then the proof of Lemma \ref{Lemma zur Bedingung C3} does not work, 
		because $ \tilde{x} $ might not belong to the domain of $ T $.
	\end{remark}
	
	The following example shows that Inequality (\ref{Ungleichung zur Symmetrie des Spektrums}) above is optimal. 
	\begin{example} 
		\begin{enumerate}
			\item Consider the bounded self-adjoint diagonal operator
							\begin{align*}
								T = \mathrm{diag} \! \left( -1, -1/2, -1/3, -1/4, ... \right) : \ell^{2}(\mathds{N}_{0}) \rightarrow \ell^{2}(\mathds{N}_{0}) 
							\end{align*}
						and, for $ N \in \mathds{N} $, the self-adjoint diagonal operator
						\begin{align*}
							S = \mathrm{diag} ( \underbrace{-1, ..., -1}_{N ~ \mathrm{times}}, 0, ... ) : \ell^{2}(\mathds{N}_{0}) \rightarrow \ell^{2}(\mathds{N}_{0}).
						\end{align*}
						Then $ S $ is of rank $ N $, and we see that 
						\begin{align*}
							&\mathrm{dim} ~ \mathrm{Ker} \! \left( E_{(-\infty, \lambda)}(T+S) - E_{(-\infty, \lambda)}(T) - I \right) = N \\
							& \text{and} \quad \mathrm{Ker} \! \left( E_{(-\infty, \lambda)}(T+S) - E_{(-\infty, \lambda)}(T) + I \right) = \{ 0 \}
						\end{align*}
						for all $ \lambda \in (-1-1/N, -1) $.
			\item Let $ a_{0} := -1 $, $ a_{1} := -1/2 $, $ a_{2} := -1/3 $. Consider the bounded self-adjoint diagonal operator 
						\begin{align*}
							T = \mathrm{diag} \! \left( a_{0}, a_{0} + \frac{1/2}{4}, a_{1}, a_{1} + \frac{1/6}{4}, a_{0} + \frac{1/2}{5}, a_{1} + \frac{1/6}{5}, 
																			 a_{0} + \frac{1/2}{6}, ... \right)
						\end{align*}
						on $ \ell^{2}(\mathds{N}_{0}) $. Since $ |a_{0}-a_{1}| = 1/2 $ and $ |a_{1}-a_{2}| = 1/6 $, it follows that the compact self-adjoint diagonal operator
						\begin{align*}
							S = -2 \cdot \mathrm{diag} \! \left( 0, \frac{1/2}{4}, 0, \frac{1/6}{4}, \frac{1/2}{5}, \frac{1/6}{5}, \frac{1/2}{6}, ... \right)
							: \ell^{2}(\mathds{N}_{0}) \rightarrow \ell^{2}(\mathds{N}_{0}) 
						\end{align*}
						is such that
						\begin{align*}
							(+)
							\begin{cases}
								&\mathrm{dim} ~ \mathrm{Ker} \! \left( E_{(-\infty, \lambda)}(T+S) - E_{(-\infty, \lambda)}(T) - I \right) = \infty \\
								& \text{and} \quad \mathrm{Ker} \! \left( E_{(-\infty, \lambda)}(T+S) - E_{(-\infty, \lambda)}(T) + I \right) = \{ 0 \}
							\end{cases}
						\end{align*}
						for $ \lambda \in \{ -1, -1/2 \} $.
						
						Clearly, this example can be extended such that $ (+) $ holds for all $ \lambda $ in $ \{ -1, -1/2, -1/3, ... \} $.
		\end{enumerate}
	\end{example}
			
	\section{On the dimension of $ \mathrm{Ker} ~ D(\lambda) $}	\label{The third section}
	In this section, we deal with the question whether the operator $ D(\lambda) $ fulfills condition (C1) in Theorem \ref{Charakterisierung modifiziert allgemeinere Version}. 
	
	Suppose that the self-adjoint operators $ T $ and $ S $ are bounded and of finite rank, respectively. 
	We will provide a list of sufficient conditions such that the kernel of $ D(\lambda) $ is infinite dimensional for all $ \lambda $ in $ \mathds{R} $.
	
	Furthermore, we will prove that the kernel of $ D(\lambda) $ is trivial for all $ \lambda $ in the \linebreak 
	interval $ ( \min \sigma_{\mathrm{ess}}(T), \max \sigma_{\mathrm{ess}}(T) ) $ and infinite dimensional for all $ \lambda $ 
	in \linebreak $ \mathds{R} \setminus [ \min \sigma_{\mathrm{ess}}(T), \max \sigma_{\mathrm{ess}}(T) ] $, provided that 
	$ S = \langle \cdot, \varphi \rangle \varphi $ is a rank one operator and the vector $ \varphi $ is cyclic for $ T $, see Theorem \ref{Main theorem II} below.
	
	\subsection{Sufficient conditions such that $ \mathrm{dim} ~ \mathrm{Ker} ~ D(\lambda) = \infty $.}
	
	Let $ \lambda \in \mathds{R} $. If the kernel of $ D(\lambda) = E_{(-\infty, \lambda)}(T+S) - E_{(-\infty, \lambda)}(T) $ is infinite dimensional, 
	then $ D(\lambda) $ fulfills conditions (C1) and (C2) in Theorem \ref{Charakterisierung modifiziert allgemeinere Version}. 
	
	Let $ N \in \mathds{N} $ be the rank of $ S $. 
	The following proposition provides a list of sufficient conditions such that the kernel of $ D(\lambda) $ is infinite dimensional.
	
	\begin{proposition} \label{Einige hinreichende Bedingungen}
		If at least one of the following three cases occurs for $ X = T $ or for $ X = T + S $, then the operator $ D(\lambda) $ is unitarily equivalent to a 
		bounded self-adjoint block-Hankel operator of order $ N $ with infinite dimensional kernel for all $ \lambda \in \mathds{R} $.
		\begin{enumerate}
			\item \label{EW mit unendlicher Vielfachheit} 
						The spectrum of $ X $ contains an eigenvalue of infinite multiplicity. 
						In particular, this pertains to the case when the range of $ X $ is finite dimensional.
			\item The spectrum of $ X $ contains infinitely many eigenvalues with multiplicity at least $ N+1 $.
			\item \label{Vielfachheit groesser eins im stetigen Spektrum}
						The spectrum of the restricted operator $ \! \left. X \right|_{\mathfrak{E}^{\bot}} $ has multiplicity at least $ N+1 $ (not necessarily uniform), 
						where $ \mathfrak{E} := \left\{ x \in \mathfrak{H} : x \text{ is an eigenvector of } X \right\} $.
		\end{enumerate}
	\end{proposition}
	
	\begin{proof}
		By Lemma \ref{Lemma zur Bedingung C3}, we know that condition (C3)$ _{N} $ in Theorem \ref{Charakterisierung modifiziert allgemeinere Version} 
		holds true for all $ \lambda \in \mathds{R} $. 
		It remains to show that $ \mathrm{dim} ~ \mathrm{Ker} ~ D(\lambda) = \infty $ for all $ \lambda \in \mathds{R} $.
		  
		First, suppose that there exists an eigenvalue $ \lambda_{0} $ of $ X=T $ with multiplicity $ m \geq N+1 $, i.\,e. $ m \in \{ N+1,N+2, ... \} \cup \{ \infty \} $. 
		Define
		\begin{equation*}
			\mathfrak{M} := \left( \mathrm{Ran} ~ E_{ \{ \lambda_{0} \} }(T) \right) \cap (\mathrm{Ran} ~ S)^{\bot} \neq \{ 0 \}. 
		\end{equation*}
		It is easy to show that $ \mathfrak{M} $ is a closed subspace of $ \mathfrak{H} $ such that
		\begin{itemize}
			\item $ \mathrm{dim} ~ \mathfrak{M} \geq m-N $,
			\item $ \left. T \right|_{\mathfrak{M}} = \left. \left( T + S \right) \right|_{\mathfrak{M}} $,
			\item $ T(\mathfrak{M}) \subset \mathfrak{M} \quad \text{and} \quad T(\mathfrak{M}^{\bot}) \subset \mathfrak{M}^{\bot} $,
			\item $ \left( T + S \right)(\mathfrak{M}) \subset \mathfrak{M} \quad \text{and} \quad \left( T + S \right)(\mathfrak{M}^{\bot}) \subset \mathfrak{M}^{\bot} $.
		\end{itemize}
		Therefore, $ \mathfrak{M} $ is contained in the kernel of $ D(\lambda) $ for all $ \lambda \in \mathds{R} $.
		
		It follows that the kernel of $ D(\lambda) $ is infinite dimensional for all $ \lambda \in \mathds{R} $ whenever case (1) or case (2) occur for the operator $ X=T $; 
		in the case when $ X = T + S $ the proof runs analogously.
	
		Now suppose that case (3) occurs for $ X = T $. Write 
		\begin{equation*}
			S = \sum_{k=1}^{N} \alpha_{k} \langle \cdot, \varphi_{k} \rangle \varphi_{k} : \mathfrak{H} \rightarrow \mathfrak{H},
		\end{equation*}
		where $ \varphi_{1}, ..., \varphi_{N} $ form an orthonormal system in $ \mathfrak{H} $ and $ \alpha_{1}, ..., \alpha_{N} $ are nonzero real numbers. 
		Define the closed subspace $ \mathfrak{N} := \overline{\mathrm{span}} \left\{ T^{j} \varphi_{k} : j \in \mathds{N}_{0}, k = 1,...,N \right\} $ of $ \mathfrak{H} $. 
		It is well known that 
		\begin{itemize}
			\item $ \left. T \right|_{\mathfrak{N}^{\bot}} = \left. \left( T + S \right) \right|_{\mathfrak{N}^{\bot}} $,
			\item $ T(\mathfrak{N}) \subset \mathfrak{N} \quad \text{and} \quad T(\mathfrak{N}^{\bot}) \subset \mathfrak{N}^{\bot} $,
			\item $ \left( T + S \right)(\mathfrak{N}) \subset \mathfrak{N} \quad \text{and} \quad \left( T + S \right)(\mathfrak{N}^{\bot}) \subset \mathfrak{N}^{\bot} $.
		\end{itemize}
		Therefore, $ \mathfrak{N}^{\bot} $ is contained in the kernel of $ D(\lambda) $ for all $ \lambda \in \mathds{R} $.
		A standard proof using the theory of direct integrals (see \cite[Chapter 7]{Birman_Solomyak_II}, see in particular \cite[Theorem 1, p.\,177]{Birman_Solomyak_II}) 
		shows that $ \mathfrak{N}^{\bot} $ is infinite dimensional.
		
		If $ X = T + S $, then the proof runs analogously.
	
		Now the proof is complete.
	\end{proof}
	
	\subsection{The case when $ S $ is a rank one operator.}	\label{Reduktion zum zyklischen Fall}
	For the rest of this section, let us assume that $ S = \langle \cdot, \varphi \rangle \varphi $ is a rank one operator.	
		
	The following example illustrates that $ \mathrm{dim} ~ \mathrm{Ker} ~ D(\lambda) $ may attain every value in $ \mathds{N} $, provided that $ \varphi $ is not cyclic for $ T $. 
	Recall that when $ \mathrm{dim} ~ \mathrm{Ker} ~ D(\lambda) $ is neither zero nor infinity, Theorem \ref{Charakterisierung modifiziert} shows that 
	Question \ref{The main question} has to be answered negatively.
	\begin{example} \label{via Krein}
		Essentially, this is an application of Kre{\u\i}n's example from \cite[pp.\,622--624]{Krein}. \newline
		Let $ 0 < \lambda < 1 $. Consider the bounded self-adjoint integral operators $ A_{j} $, $ j = 0,1 $, with kernel functions
		\begin{equation*} 
			a_{0}(x,y) = \begin{cases}
  									 \sinh(x) \mathrm{e}^{-y}  & \text{if } x \leq y \\
   									 \sinh(y) \mathrm{e}^{-x} & \text{if } x \geq y
 									 \end{cases}
 			~~ \text{and} ~~
			a_{1}(x,y) = \begin{cases}
  									 \cosh(x) \mathrm{e}^{-y}  & \text{if } x \leq y \\
   									 \cosh(y) \mathrm{e}^{-x} & \text{if } x \geq y
 									 \end{cases}
		\end{equation*}
		on the Hilbert space $ L^{2}(0,\infty) $. 
		By \cite[pp.\,622--624]{Krein}, we know that $ A_{0} - A_{1} $ is of rank one and that 
		the difference $ E_{(-\infty, \lambda)}(A_{0}) - E_{(-\infty, \lambda)}(A_{1}) $ is a Hankel operator. 		
		Furthermore, it was shown in \cite[Theorem 1]{Kostrykin} that $ E_{(-\infty, \lambda)}(A_{0}) - E_{(-\infty, \lambda)}(A_{1}) $ has a simple purely absolutely continuous spectrum 
		filling in the interval $ [-1,1] $. 
		In particular, the kernel of 
		 \begin{equation*}
		 	E_{(-\infty, \lambda)}(A_{0}) - E_{(-\infty, \lambda)}(A_{1})
		 \end{equation*}
		 is trivial. Let $ k \in \mathds{N} $. Now consider block diagonal operators
		 \begin{align*}
		 	\widetilde{A}_{j} := A_{j} \oplus M : L^{2}(0,\infty) \oplus \mathds{C}^{k} \rightarrow L^{2}(0,\infty) \oplus \mathds{C}^{k}, \quad j = 0,1,
		 \end{align*}
		 where $ M \in \mathds{C}^{k \times k} $ is an arbitrary fixed self-adjoint matrix. Then one has
		 \begin{align*}
		 	\mathrm{dim} ~ \mathrm{Ker} \left( E_{(-\infty, \lambda)}(\widetilde{A}_{0}) - E_{(-\infty, \lambda)}(\widetilde{A}_{1}) \right) = k.
		 \end{align*}
	\end{example}
	
	The following consideration shows that (up to at most countably many $ \lambda $ in the essential spectrum of $ T $) this is the only type of counterexample 
	to Question \ref{The main question} above.
		
	The closed subspace $ \mathfrak{N}^{\bot} $ of $ \mathfrak{H} $ might be trivial, finite dimensional, or infinite dimensional, 
	where $ \mathfrak{N} := \overline{\mathrm{span}} \{ T^{j} \varphi : j \in \mathds{N}_{0} \} $.
	\begin{itemize}	
		\item[\textbf{Case 1.}] 
					If $ \mathfrak{N}^{\bot} $ is trivial, then $ \varphi $ is cyclic for $ T $ and Proposition \ref{Main result} below implies that $ D(\lambda) $ 
					is unitarily equivalent to a bounded self-adjoint Hankel operator for all $ \lambda $ in $ \mathds{R} $ except for at most countably many $ \lambda $ 
					in $ \sigma_{\mathrm{ess}}(T) $.
		\item[\textbf{Case 2.}] 
					Suppose that $ \mathrm{dim} ~ (\mathfrak{N}^{\bot}) =: k \in \mathds{N} $. We will reduce this situation to the first case.
					Let us identify $ \mathfrak{N}^{\bot} $ with $ \mathds{C}^{k} $. 
					The restricted operators $ \left. T \right|_{\mathfrak{N}^{\bot}} $ and $ \left. (T+S) \right|_{\mathfrak{N}^{\bot}} $ coincide on $ \mathfrak{N}^{\bot} $, 
					and since $ \mathfrak{N} $ reduces both $ T $ and $ T+S $ there exists a self-adjoint matrix $ M $ in $ \mathds{C}^{k \times k} $ such that $ T $ and $ T+S $ can be 
					identified with the block diagonal operators $ \left. T \right|_{\mathfrak{N}} \oplus M $ and $ \left. (T+S) \right|_{\mathfrak{N}} \oplus M $ 
					acting on $ \mathfrak{N} \oplus \mathds{C}^{k} $, respectively. Therefore, 
					\begin{equation*}
						D(\lambda)
						= \Big( E_{(-\infty, \lambda)} \big( \! \left. T \right|_{\mathfrak{N}} + \left. S \right|_{\mathfrak{N}} \big)
							- E_{(-\infty, \lambda)} \big( \! \left. T \right|_{\mathfrak{N}} \big) \Big)
							\oplus 0,	\quad \lambda \in \mathds{R},
					\end{equation*}
					and $ \varphi $ is cyclic for $ \left. T \right|_{\mathfrak{N}} $.
		\item[\textbf{Case 3.}] 
					Since $ \mathfrak{N}^{\bot} \subset \mathrm{Ker} ~ D(\lambda) $ for all $ \lambda $ in $ \mathds{R} $, 
					it follows from Lemma \ref{Lemma zur Bedingung C3} and Theorem \ref{Charakterisierung modifiziert} that $ D(\lambda) $ is unitarily equivalent to a bounded 
					self-adjoint Hankel operator for all $ \lambda $ in $ \mathds{R} $ if $ \mathfrak{N}^{\bot} $ is infinite dimensional.
	\end{itemize}
		
	\subsection{The case when $ \varphi $ is cyclic for $ T $.}	\label{The first subsection}
	This subsection is devoted to the proof of the following theorem:  
	\begin{theorem} \label{Main theorem II}
		Suppose that the self-adjoint operators $ T $ and $ S = \langle \cdot, \varphi \rangle \varphi $ are bounded and of rank one, respectively, 
		and that the vector $ \varphi $ is cyclic for $ T $. 
		Let $ \lambda \in \mathds{R} \setminus \{ \min \sigma_{\mathrm{ess}}(T), ~ \max \sigma_{\mathrm{ess}}(T) \} $. ~Then the kernel of $ D(\lambda) $ is
		\begin{enumerate}
			\item infinite dimensional if and only if $ \lambda \in \mathds{R} \setminus \left[ \min \sigma_{\mathrm{ess}}(T), \max \sigma_{\mathrm{ess}}(T) \right] $.
			\item trivial if and only if $ \lambda \in \left( \min \sigma_{\mathrm{ess}}(T), \max \sigma_{\mathrm{ess}}(T) \right) $.
		\end{enumerate}
		In particular, one has	
		\begin{align*}
			\text{either} \quad \mathrm{Ker} ~ D(\lambda) = \{ 0 \} \quad \text{or} \quad \mathrm{dim} ~ \mathrm{Ker} ~ D(\lambda) = \infty.
		\end{align*}
	\end{theorem}
	The proof is based on a result by Liaw and Treil \cite{Treil_und_Liaw} and some harmonic analysis.

	Theorem \ref{Main theorem II} will be an important ingredient in the proof of Proposition \ref{Main result} below. Likewise, it is of independent interest.
	Note that, according to Theorem \ref{Main theorem II}, the kernel of $ D(\lambda) $ is trivial for \emph{all} $ \lambda $ between $ \min \sigma_{\mathrm{ess}}(T) $ 
	and $ \max \sigma_{\mathrm{ess}}(T) $, no matter if the interval $ (\min \sigma_{\mathrm{ess}}(T), \max \sigma_{\mathrm{ess}}(T)) $ contains points from the resolvent set of 
	$ T $, isolated eigenvalues of $ T $, etc.
	
	It will be useful to write $ S = S_{\alpha} = \alpha \langle \cdot, \varphi \rangle \varphi $ 
	for some real number $ \alpha \neq 0 $ such that $ \| \varphi \| = 1 $.
	
	Let $ \lambda \in \mathds{R} $. Again, we write 
	\begin{equation*}
		P_{\lambda} = E_{(-\infty, \lambda)}(T+S_{\alpha}) \quad \text{and} \quad Q_{\lambda} = E_{(-\infty, \lambda)}(T).
	\end{equation*}
	Observe that the kernel of $ P_{\lambda} - Q_{\lambda} $ is equal to the orthogonal sum 
	of $ \left( \mathrm{Ran} ~ P_{\lambda} \right) \cap \left( \mathrm{Ran} ~ Q_{\lambda} \right) $ 
	and $ \left( \mathrm{Ker} ~ P_{\lambda} \right) \cap \left( \mathrm{Ker} ~ Q_{\lambda} \right) $. 
	Therefore, we will investigate the dimensions of $ \left( \mathrm{Ran} ~ P_{\lambda} \right) \cap \left( \mathrm{Ran} ~ Q_{\lambda} \right) $ 
	and $ \left( \mathrm{Ker} ~ P_{\lambda} \right) \cap \left( \mathrm{Ker} ~ Q_{\lambda} \right) $ separately.
	
	Now we follow \cite[pp.\,1948--1949]{Treil_und_Liaw} in order to represent the operators $ T $ and $ T+S_{\alpha} $ such that \cite[Theorem 2.1]{Treil_und_Liaw} is applicable.
	
	Define Borel probability measures $ \bbmu $ and $ \bbmu_{\alpha} $ on $ \mathds{R} $ by
	\begin{align*}
		\bbmu(\Omega) := \langle E_{\Omega}(T) \varphi, \varphi \rangle \quad \text{and} \quad 
		\bbmu_{\alpha}(\Omega) := \langle E_{\Omega}(T+S_{\alpha}) \varphi, \varphi \rangle, \quad
		\Omega \in \mathcal{B}(\mathds{R}),
	\end{align*}
	respectively. According to \cite[Proposition 5.18]{Schmuedgen}, there exist unitary operators $ U : \mathfrak{H} \rightarrow L^{2}(\bbmu) $ and 
	$ U_{\alpha} : \mathfrak{H} \rightarrow L^{2}(\bbmu_{\alpha}) $ such that $	U	T U^{\ast} = M_{t} $ is the multiplication operator by the independent variable on $ L^{2}(\bbmu) $, 
	$ U_{\alpha} (T+S_{\alpha}) U_{\alpha}^{\ast} = M_{s} $ is the multiplication operator by the independent variable on $ L^{2}(\bbmu_{\alpha}) $, and one has both 
	$ (U \varphi)(t) = 1 $ on $ \mathds{R} $ and $ (U_{\alpha} \varphi)(s) = 1 $ on $ \mathds{R} $.	
	Clearly, the operators $ U $ and $ U_{\alpha} $ are uniquely determined by these properties.
	By \cite[Theorem 2.1]{Treil_und_Liaw}, the unitary operator $ V_{\alpha} := U_{\alpha} U^{\ast} : L^{2}(\bbmu) \rightarrow L^{2}(\bbmu_{\alpha}) $ is given by
	\begin{align}	\label{Der unitaere Operator von Treil_und_Liaw}
		\left( V_{\alpha} f \right)(x) = f(x) - \alpha \int \frac{f(x) - f(t)}{x - t} \mathrm{d} \bbmu(t)
	\end{align}
	for all continuously differentiable functions $ f : \mathds{R} \rightarrow \mathds{C} $ with compact support. 
	For the rest of this subsection, we suppose that $ V_{\alpha} $ satisfies (\ref{Der unitaere Operator von Treil_und_Liaw}). 
	Without loss of generality, we may further assume that $ T $ is already the multiplication operator by the independent variable on $ L^{2}(\bbmu) $, 
	i.\,e., we identify $ \mathfrak{H} $ with $ L^{2}(\bbmu) $, $ T $ with $ U T U^{\ast} $, and $ T+S_{\alpha} $ with $ U	(T+S_{\alpha}) U^{\ast} $.
	
	In order to prove Theorem \ref{Main theorem II}, we need the following lemma.
	\begin{lemma} \label{Main theorem II part I}
		Let $ \lambda \in \mathds{R} \setminus \{ \max \sigma_{\mathrm{ess}}(T) \} $. 
		Then one has that the dimension of $ \left( \mathrm{Ran} ~ P_{\lambda} \right) \cap \left( \mathrm{Ran} ~ Q_{\lambda} \right) $ is
		\begin{enumerate}
			\item infinite if and only if $ \lambda > \max \sigma_{\mathrm{ess}}(T) $.
			\item zero if and only if $ \lambda < \max \sigma_{\mathrm{ess}}(T) $.
		\end{enumerate}
	\end{lemma}
	
	\begin{proof}	
		The idea of this proof is essentially due to the author's supervisor, Vadim Kostrykin.
		
		The well-known fact (see, e.\,g., \cite[Example 5.4]{Schmuedgen}) that $ \mathrm{supp} ~ \bbmu_{\alpha} = \sigma(T + S_{\alpha}) $ implies that the cardinality of 
		$ (\lambda, \infty) \cap \mathrm{supp} ~ \bbmu_{\alpha} $ is infinite [resp.\,finite] 
		if and only if $ \lambda < \max \sigma_{\mathrm{ess}}(T) $ [resp.\,$ \lambda > \max \sigma_{\mathrm{ess}}(T) $]. \newline
		\textbf{Case 1.} 
		The cardinality of $ (\lambda, \infty) \cap \mathrm{supp} ~ \bbmu_{\alpha} $ is finite. \newline
		Since $ \lambda > \max \sigma_{\mathrm{ess}}(T) $, it follows that
		\begin{align*}
			\mathrm{dim} ~ \mathrm{Ran} ~ E_{[\lambda, \infty)}(T+S_{\alpha}) < \infty \quad \text{and} \quad \mathrm{dim} ~ \mathrm{Ran} ~ E_{[\lambda, \infty)}(T) < \infty.
		\end{align*} 
		Therefore, $ \mathrm{Ran} ~ E_{(-\infty,\lambda)}(T+S_{\alpha}) \cap \mathrm{Ran} ~ E_{(-\infty,\lambda)}(T) $ is infinite dimensional. \newline
		\textbf{Case 2.} 
		The cardinality of $ (\lambda, \infty) \cap \mathrm{supp} ~ \bbmu_{\alpha} $ is infinite. \newline
		If $ \lambda \leq \min \sigma(T) $ or $ \lambda \leq \min \sigma(T + S_{\alpha}) $, 
		then $ \left( \mathrm{Ran} ~ P_{\lambda} \right) \cap \left( \mathrm{Ran} ~ Q_{\lambda} \right) = \{ 0 \} $, as claimed. 
		Now suppose that $ \lambda > \min \sigma(T) $ and $ \lambda > \min \sigma(T + S_{\alpha}) $.
		
		Let $ f \in \left( \mathrm{Ran} ~ P_{\lambda} \right) \cap \left( \mathrm{Ran} ~ Q_{\lambda} \right) $. Then one has
		\begin{align*}
			f(x) = 0 ~ \text{for } \bbmu \text{-almost all } x \geq \lambda \quad \text{and} \quad 
			\left( V_{\alpha}f \right)(x) = 0 ~ \text{for } \bbmu_{\alpha} \text{-almost all } x \geq \lambda.
		\end{align*}
		Choose a representative $ \tilde{f} $ in the equivalence class of $ f $ such that $ \tilde{f}(x) = 0 $ for \textit{all} $ x \geq \lambda $. 
		Let $ r \in \Big( 0, \frac{ \max \sigma_{\mathrm{ess}}(T) - \lambda }{  3 } \Big) $. 
		According to \cite[Corollary 6.4 (a)]{Knapp} and the fact that $ \bbmu $ is a finite Borel measure on $ \mathds{R} $, we know that the set of continuously differentiable 
		scalar-valued functions on $ \mathds{R} $ with compact support is dense in $ L^{2}(\bbmu) $ with respect to $ \| \cdot \|_{L^{2}(\bbmu)} $. 
		Thus, a standard mollifier argument shows that we can choose continuously differentiable functions $ \tilde{f}_{n} : \mathds{R} \rightarrow \mathds{C} $ with 
		compact support such that	
	 	\begin{align*}
	 		\big\| \tilde{f}_{n} - \tilde{f} \big\|_{L^{2}(\bbmu)} < 1/n \quad \text{and} \quad
	 		\tilde{f}_{n}(x) = 0 \text{ for all } x \geq \lambda + r, \quad n \in \mathds{N}.
	 	\end{align*}
	 	In particular, we may insert $ \tilde{f}_{n} $ into Formula (\ref{Der unitaere Operator von Treil_und_Liaw}) and obtain
	 	\begin{equation*}
	 		\left( V_{\alpha} \tilde{f}_{n} \right)(x) = \alpha \int_{(-\infty, \lambda + r)} \frac{\tilde{f}_{n}(t)}{x - t} \mathrm{d} \bbmu(t) \quad \text{for all } x \geq \lambda + 2r.
	 	\end{equation*}
	 	It is readily seen that
	 	\begin{equation*}
	 		\left( B g \right)(x) := \int_{(-\infty, \lambda + r)} \frac{g(t)}{x - t} \mathrm{d} \bbmu(t), \quad x \geq \lambda + 2r,
	 	\end{equation*}
	 	defines a bounded operator 
		$ B : L^{2} \! \left( \mathds{1}_{(-\infty, \lambda + r)} \mathrm{d} \bbmu \right) \rightarrow 
		L^{2} \! \left( \mathds{1}_{ \left[ \! \right. \left. \lambda + 2r, \infty \right) } \mathrm{d} \bbmu_{\alpha} \right) $ with operator norm $ \leq 1/r $. 
	 	It is now easy to show that
	 	\begin{equation*} \tag{$ \ast $}
	 		\int_{(-\infty, \lambda]} \frac{\tilde{f}(t)}{x - t} \mathrm{d} \bbmu(t) = 0 \quad \text{for } \bbmu_{\alpha} \text{-almost all } x \geq \lambda + 2r. 
	 	\end{equation*}
	 	As $ r \in \Big( 0, \frac{ \max \sigma_{\mathrm{ess}}(T) - \lambda }{  3 } \Big) $ in ($ \ast $) was arbitrary, we get that
	 	\begin{equation*}
	 		\int_{(-\infty, \lambda]} \frac{\tilde{f}(t)}{x - t} \mathrm{d} \bbmu(t) = 0 \quad \text{for } \bbmu_{\alpha} \text{-almost all } x > \lambda. 
	 	\end{equation*}
 		From now on, we may assume without loss of generality that $ \tilde{f} $ is real-valued.

		Consider the holomorphic function from $ \mathds{C} \setminus (-\infty, \lambda] $ to $ \mathds{C} $ defined by
	 	\begin{equation*}
	 		z \mapsto \int_{(-\infty, \lambda]} \frac{\tilde{f}(t)}{z - t} \mathrm{d} \bbmu(t).
	 	\end{equation*}
		Since $ \lambda < \max \sigma_{\mathrm{ess}}(T) $, the identity theorem for holomorphic functions implies that
		\begin{equation*}
			\int_{(-\infty, \lambda]} \frac{\tilde{f}(t)}{z - t} \mathrm{d} \bbmu(t) = 0 \quad \text{for all } z \in \mathds{C} \setminus (-\infty, \lambda].
		\end{equation*}
		This yields
		\begin{equation*} \tag{$ \ast \ast $}
			\int_{(-\infty, \lambda]} \frac{\tilde{f}(t)}{(x - t)^{2} + y^{2}} \mathrm{d} \bbmu(t) = 0 \quad \text{for all } x \in \mathds{R}, ~ y > 0.
		\end{equation*}
		Consider the positive finite Borel measure $ \nu_{1} : \mathcal{B}(\mathds{R}) \rightarrow [0, \infty) $ and 
		the finite signed Borel measure $ \nu_{2} : \mathcal{B}(\mathds{R}) \rightarrow \mathds{R} $ defined by
		\begin{align*}
			\nu_{1}(\Omega) := \int_{\Omega \cap (-\infty, \lambda]} \mathrm{d} \bbmu(t), \quad 
			\nu_{2}(\Omega) := \int_{\Omega \cap (-\infty, \lambda]} \tilde{f}(t) \mathrm{d} \bbmu(t);
		\end{align*}
		note that $ \tilde{f} $ belongs to $ L^{1}(\bbmu) $.
		
		Denote by $ p_{\nu_{j}} : \{ x + \mathrm{i} y \in \mathds{C} : x \in \mathds{R} ,y > 0 \} \rightarrow \mathds{R} $ the Poisson transform of $ \nu_{j} $,
		\begin{align*}
			p_{\nu_{j}}(x + \mathrm{i} y) := y \int_{\mathds{R}} \frac{\mathrm{d} \nu_{j}(t)}{(x - t)^{2} + y^{2}}, \quad x \in \mathds{R}, ~ y > 0, ~ j=1,2.
		\end{align*}
		It follows from ($ \ast \ast $) that
		\begin{equation*}
			p_{\nu_{2}}(x + \mathrm{i} y) = 0 \quad \text{for all } x \in \mathds{R}, ~ y > 0.
		\end{equation*}
		Furthermore, since $ \nu_{1} $ is not the trivial measure, one has
			\begin{equation*}
				p_{\nu_{1}}(x + \mathrm{i} y) > 0 \quad \text{for all } x \in \mathds{R}, ~ y > 0.
			\end{equation*}
		Now \cite[Proposition 2.2]{Jaksic_Last} implies that
		\begin{equation*}
			0 = \lim_{y \searrow 0} \frac{p_{\nu_{2}}(x + \mathrm{i} y)}{p_{\nu_{1}}(x + \mathrm{i} y)} = \tilde{f}(x) \quad \text{for } \bbmu \text{-almost all } x \leq \lambda. 
		\end{equation*}
		Hence $ \tilde{f}(x) = 0 $ for $ \bbmu $-almost all $ x \in \mathds{R} $. 
		We conclude that $ \left( \mathrm{Ran} ~ P_{\lambda} \right) \cap \left( \mathrm{Ran} ~ Q_{\lambda} \right) $ is trivial. This finishes the proof.
	\end{proof}
	
	Analogously, one shows that the following lemma holds true.	\newpage
	\begin{lemma} \label{Main theorem II part II}
		Let $ \lambda \in \mathds{R} \setminus \{ \min \sigma_{\mathrm{ess}}(T) \} $. 
		Then one has that the dimension of $ \left( \mathrm{Ker} ~ P_{\lambda} \right) \cap \left( \mathrm{Ker} ~ Q_{\lambda} \right) $ is
		\begin{enumerate}
			\item infinite if and only if $ \lambda < \min \sigma_{\mathrm{ess}}(T) $.
			\item zero if and only if $ \lambda > \min \sigma_{\mathrm{ess}}(T) $.
		\end{enumerate}
	\end{lemma}
	
	\begin{remark}
		The proof of Lemma \ref{Main theorem II part I} does not work if $ T $ is unbounded. 
		To see this, consider the case where the essential spectrum of $ T $ is bounded from above 
		and $ T $ has infinitely many isolated eigenvalues greater than $ \max \sigma_{\mathrm{ess}}(T) $.
	\end{remark}
	
	\begin{proof}[Proof of Theorem \ref{Main theorem II}]
		Taken together, Lemmas \ref{Main theorem II part I} and \ref{Main theorem II part II} imply Theorem \ref{Main theorem II}.
	\end{proof}

	\section{On non-invertibility of $ D(\lambda) $} \label{The section with finite rank perturbation}
	
	In this section, the self-adjoint operator $ T $ is assumed to be bounded. The main purpose of this section is to establish the following theorem.
	\begin{theorem}	\label{Main theorem III}
		Let $ S : \mathfrak{H} \rightarrow \mathfrak{H} $ be a compact self-adjoint operator. 
		Then the following assertions hold true.
		\begin{enumerate}
			\item If $ \lambda \in \mathds{R} \setminus \sigma_{\mathrm{ess}}(T) $, then $ D(\lambda) $ is a compact operator. 
						In particular, zero belongs to the essential spectrum of $ D(\lambda) $.
			\item Zero belongs to the essential spectrum of $ D(\lambda) $ for all but at most countably many $ \lambda $ in $ \sigma_{\mathrm{ess}}(T) $.	\label{part two}
		\end{enumerate}
	\end{theorem}
	Note that we cannot exclude the case that the exceptional set is dense in $ \sigma_{\mathrm{ess}}(T) $.
	
	\begin{remark}
	Mart{\'{\i}}nez-Avenda{\~n}o and Treil have shown ``that given any compact subset of the complex plane containing zero, 
	there exists a Hankel operator having this set as its spectrum'' (see \cite[p.\,83]{Treil_Martinez}).
	Thus, Theorem \ref{Main theorem III} and \cite[Theorem 1.1]{Treil_Martinez} lead to the following result:
	
	\textit{for all $ \lambda $ in $ \mathds{R} $ except for at most countably many $ \lambda $ in $ \sigma_{\mathrm{ess}}(T) $, there exists a Hankel operator $ \Gamma_{\lambda} $ 
	such that $ \sigma(\Gamma_{\lambda}) = \sigma \big( D(\lambda) \big) $.}
	\end{remark}
		
	First, we will prove Theorem \ref{Main theorem III} in the case when the range of $ S $ is finite dimensional. If $ S $ is compact and the range of $ S $ is infinite dimensional, 
	then the proof has to be modified.
		
	\subsection{The case when the range of $ S $ is finite dimensional}
	
	Throughout this subsection, we consider a self-adjoint finite rank operator 
	\begin{equation*}
		S = \sum_{j=1}^{N} \alpha_{j} \langle \cdot, \varphi_{j} \rangle \varphi_{j} : \mathfrak{H} \rightarrow \mathfrak{H}, \quad N \in \mathds{N},
	\end{equation*}
	where $ \varphi_{1}, ..., \varphi_{N} $ form an orthonormal system in $ \mathfrak{H} $ and $ \alpha_{1}, ..., \alpha_{N} $ are nonzero real numbers.
	
	Note that if there exists $ \lambda_{0} $ in $ \mathds{R} $ such that	
	\begin{equation*}
		\mathrm{dim} ~ \mathrm{Ran} ~ E_{ \{ \lambda_{0} \} }(T) = \infty \quad \text{or} \quad \mathrm{dim} ~ \mathrm{Ran} ~ E_{ \{ \lambda_{0} \} }(T+S) = \infty,
	\end{equation*}
	then $ \mathrm{dim} ~ \mathrm{Ker} ~ D(\lambda) = \infty $ 
	(see the proof of Proposition \ref{Einige hinreichende Bedingungen} (\ref{EW mit unendlicher Vielfachheit}) above) 
	and hence $ 0 \in \sigma_{\mathrm{ess}} \big( D(\lambda) \big) $ for all $ \lambda \in \mathds{R} $.  
		
	Define the sets $ \mathcal{M}(X) $, $ \mathcal{M}_{-}(X) $, and $ \mathcal{M}_{+}(X) $ by
	\begin{align*}
		&\mathcal{M}(X) := \{ \lambda \in \sigma_{\mathrm{ess}}(X) : \text{there exist } \lambda_{k}^{\pm} \text{ in } \sigma(X) \text{ such that } 
										\lambda_{k}^{-} \nearrow \lambda, ~ \lambda_{k}^{+} \searrow \lambda \}, \\
		&\mathcal{M}_{-}(X) 
			:= \{ \lambda \in \sigma_{\mathrm{ess}}(X) : \text{there exist } \lambda_{k}^{-} \text{ in } \sigma(X) \text{ such that } \lambda_{k}^{-} \nearrow \lambda \} 
				\setminus \mathcal{M}(X), \\
		&\mathcal{M}_{+}(X) 
			:= \{ \lambda \in \sigma_{\mathrm{ess}}(X) : \text{there exist } \lambda_{k}^{+} \text{ in } \sigma(X) \text{ such that } \lambda_{k}^{+} \searrow \lambda \} 
				\setminus \mathcal{M}(X),
	\end{align*}
	where $ X = T $ or $ X = T+S $. The following well-known result shows that these sets do not depend on whether $ X = T $ or $ X = T + S $. 
	\begin{lemma}[see {\cite[Proposition 2.1]{Arazy_Zelenko}}; see also {\cite[p.\,83]{Behncke}}]
		Let $ A $ and $ B $ be bounded self-adjoint operators acting on $ \mathfrak{H} $. 
		If $ N := \mathrm{dim} ~ \mathrm{Ran} ~ B $ is in $ \mathds{N} $ and $ \mathcal{I} \subset \mathds{R} $ is a nonempty interval contained in the resolvent set of $ A $, 
		then $ \mathcal{I} $ contains no more than $ N $ eigenvalues of the operator $ A + B $ (taking into account their multiplicities).
	\end{lemma}
	
	In view of this lemma and the fact that the essential spectrum is invariant under compact perturbations, we will write $ \mathcal{M} $ instead of $ \mathcal{M}(X) $, 
	$ \mathcal{M}_{+} $ instead of $ \mathcal{M}_{+}(X) $, and $ \mathcal{M}_{-} $ instead of $ \mathcal{M}_{-}(X) $, where $ X = T $ or $ X = T+S $.
	
	\begin{lemma} \label{Trace Class Ergebnis}
		Let $ \lambda \in \mathds{R} \setminus ( \mathcal{M} \cup \mathcal{M}_{-} ) $. Then $ D(\lambda) $ is a trace class operator. 
	\end{lemma}
	
	\begin{proof}
		There exists an infinitely differentiable function $ \psi : \mathds{R} \rightarrow \mathds{R} $ with compact support such that
		\begin{equation*}
			E_{(-\infty, \lambda)}(T+S) - E_{(-\infty, \lambda)}(T)
			= \psi (T+S) - \psi (T).
		\end{equation*}
		Combine \cite[p.\,156, Equation (8.3)]{Birman_Solomyak} with \cite[p.\,532]{Peller_II} and \cite[Theorem 2]{Peller_II}, and it follows that $ D(\lambda) $ is 
		a trace class operator. 
	\end{proof}
	An analogous proof shows that $ D(\lambda) $ is a trace class operator for $ \lambda $ in $ \mathcal{M}_{-} $, 
	provided that $ E_{\{ \lambda \}}(T+S) - E_{\{ \lambda \}}(T) $ is of trace class.
	
	\begin{proposition} \label{Null im Spektrum bis auf abz}
		One has $ 0 \in \sigma_{\mathrm{ess}} \big( D(\lambda) \big) $ for all but at most countably many $ \lambda \in \mathds{R} $. 
	\end{proposition}
	
	In the proof of Proposition \ref{Null im Spektrum bis auf abz}, we will use the notion of weak convergence for sequences of probability measures.
	\begin{definition}
		Let $ \mathcal{E} $ be a metric space. 
		A sequence $ \bbnu_{1}, \bbnu_{2}, ... $ of Borel probability measures on $ \mathcal{E} $ is said to converge \emph{weakly} to 
		a Borel probability measure $ \bbnu $ on $ \mathcal{E} $ if
		\begin{align*}
			\lim_{n \rightarrow \infty} \int f \mathrm{d} \bbnu_{n} 
			= \int f \mathrm{d} \bbnu \quad \text{for every bounded continuous function } f : \mathcal{E} \rightarrow \mathds{R}.
		\end{align*}
		If $ \bbnu_{1}, \bbnu_{2}, ... $ converges weakly to $ \bbnu $, then we shall write $ \bbnu_{n} \stackrel{w}{\rightarrow} \bbnu $, $ n \rightarrow \infty $.
	\end{definition}

	\begin{proof}[Proof of Proposition \ref{Null im Spektrum bis auf abz}]
		First, we note that if $ \lambda < \min \big( \sigma(T) \cup \sigma(T+S) \big) $ or $ \lambda > \max \big( \sigma(T) \cup \sigma(T+S) \big) $, then 
		$ D(\lambda) $ is the zero operator, and there is nothing to show. So let us henceforth assume
		that $ \lambda \geq \min \big( \sigma(T) \cup \sigma(T+S) \big) $ and $ \lambda \leq \max \big( \sigma(T) \cup \sigma(T+S) \big) $.	
		
		The idea of the proof is to apply Weyl's criterion (see, e.\,g., \cite[Proposition 8.11]{Schmuedgen}) to a suitable sequence of normed vectors.		
		In this proof, we denote by $ \| g \|_{\infty, \, \mathcal{K}} $ the supremum norm of a function $ g : \mathcal{K} \rightarrow \mathds{R} $, where 
		$ \mathcal{K} $ is a compact subset of $ \mathds{R} $, and 
		by $ \| A \|_{\mathrm{op}} $ the usual operator norm of an operator $ A : \mathfrak{H} \rightarrow \mathfrak{H} $. 
	
		Choose a sequence $ (x_{n})_{n \in \mathds{N}} $ of normed vectors in $ \mathfrak{H} $ such that
		\begin{align*}
			&x_{1} ~ \bot ~ \left\{ \varphi_{k} : k = 1, ..., N  \right\}, ~ x_{2} ~ \bot ~ \left\{ x_{1}, \varphi_{k}, T \varphi_{k} : k = 1, ..., N \right\}, ~ ..., \\
			&x_{n} ~ \bot ~ \left\{ x_{1}, ..., x_{n-1}, T^{j} \varphi_{k} : j \in \mathds{N}_{0}, ~ j \leq n-1, ~ k = 1, ..., N \right\}, ~ ...
		\end{align*}
		Consider sequences of Borel probability measures $ (\bbnu_{n})_{n \in \mathds{N}} $ and $ (\tilde{\bbnu}_{n})_{n \in \mathds{N}} $ that are defined as follows:
		\begin{align*}
			\bbnu_{n}(\Omega) := \langle E_{\Omega}(T) x_{n}, x_{n} \rangle, \quad 
			\tilde{\bbnu}_{n}(\Omega) := \langle E_{\Omega}(T+S) x_{n}, x_{n} \rangle, \quad \Omega \in \mathcal{B}(\mathds{R}).
		\end{align*}
		It is easy to see that by Prohorov's theorem (see, e.\,g., \cite[Proposition 7.2.3]{Parthasarathy}), there exist 
		a subsequence of a subsequence of $ (x_{n})_{n \in \mathds{N}} $ and 
		Borel probability measures $ \bbnu $ and $ \tilde{\bbnu} $ with support contained in $ \sigma(T) $ and $ \sigma(T+S) $, respectively, such that
		\begin{equation*}
			\bbnu_{n_{k}} \stackrel{w}{\rightarrow} \bbnu ~ \text{ as } ~ k \rightarrow \infty \quad \text{and} \quad 
			\tilde{\bbnu}_{n_{k_{\ell}}} \stackrel{w}{\rightarrow} \tilde{\bbnu} ~ \text{ as } ~ \ell \rightarrow \infty.
		\end{equation*}

		Due to this observation, we consider the sequences $ \big( x_{n_{k_{\ell}}} \big)_{\ell \in \mathds{N}} $, $ \big( \bbnu_{n_{k_{\ell}}} \big)_{\ell \in \mathds{N}} $, 
		and $ \big( \tilde{\bbnu}_{n_{k_{\ell}}} \big)_{\ell \in \mathds{N}} $ 
		which will be denoted again by $ (x_{n})_{n \in \mathds{N}} $, $ (\bbnu_{n})_{n \in \mathds{N}} $, and $ (\tilde{\bbnu}_{n})_{n \in \mathds{N}} $.
		
		Put $ \mathcal{N}_{T} := \{ \mu \in \mathds{R} : \bbnu( \{ \mu \} ) > 0 \} $ and $ \mathcal{N}_{T+S} := \{ \mu \in \mathds{R} : \tilde{\bbnu}( \{ \mu \} ) > 0 \} $. 
		Then the set $ \mathcal{N}_{T} \cup \mathcal{N}_{T+S} $ is at most countable. 
		Consider the case where $ \lambda $ does not belong to $ \mathcal{N}_{T} \cup \mathcal{N}_{T+S} $. 
		Define $ \xi := \min \left\{ \min \sigma(T), ~ \min \sigma(T+S) \right\} - 1 $. 
		Consider the continuous functions $ f_{m} : \mathds{R} \rightarrow \mathds{R} $, $ m \in \mathds{N} $, that are defined by
		\begin{equation*}
			f_{m}(t) := \big( 1 + m (t - \xi) \big) \cdot \mathds{1}_{\left[ \xi - 1/m, ~ \xi \right]}(t)
										+ \mathds{1}_{\left( \xi, ~ \lambda \right)} (t)
										+ \big( 1 - m (t - \lambda) \big) \cdot \mathds{1}_{\left[ \lambda, \lambda + 1/m \right]}(t).
		\end{equation*}
		The figure below shows (qualitatively) the graph of $ f_{m} $.	
		\begin{figure}[ht]
			\begin{center}
				\begin{picture}(200,50)
					\put(0,0){\line(1,0){60}}
					\put(60,0){\line(1,4){10}}
					\put(70,40){\line(1,0){60}}
					\put(130,40){\line(1,-4){10}}
					\put(140,0){\line(1,0){60}}
				\end{picture}
				\caption{The graph of $ f_{m} $.}
			\end{center}
		\end{figure}
		
		For all $ m \in \mathds{N} $, choose polynomials $ p_{m, k} $ such that
		\begin{align} \label{Approximiere stetige Funktion durch Polynome}
			\| f_{m} - p_{m, k} \|_{\infty, \, \mathcal{K}} \rightarrow 0 \quad \text{as } k \rightarrow \infty,
		\end{align}	
		where $ \mathcal{K} := \big[ \! \min \big( \sigma(T) \cup \sigma(T+S) \big) - 10, \, \max \big( \sigma(T) \cup \sigma(T+S) \big) + 10 \big] $.
		
		By construction of $ (x_{n})_{n \in \mathds{N}} $, one has
		\begin{align} \label{Gleichheit Polynome}
			p_{m,k}(T+S) x_{n} = p_{m, k}(T) x_{n} \quad \text{for all } n > \text{degree of } p_{m, k }.
		\end{align}
		
		For all $ m \in \mathds{N} $, the function $ |\mathds{1}_{(-\infty, \lambda)} - f_{m}|^{2} $ is bounded, measurable, and continuous except for 
		a set of both $ \bbnu $-measure zero and $ \tilde{\bbnu} $-measure zero. 
		
		Now (\ref{Gleichheit Polynome}) and the Portmanteau theorem (see, e.\,g., \cite[Theorem 13.16 (i) and (iii)]{Klenke}) imply
		\begin{align*}
			&\limsup_{n \rightarrow \infty} \left\| \left( E_{(-\infty, \lambda)}(T+S) - E_{(-\infty, \lambda)}(T) \right) x_{n} \right\| \\
			&\leq \left( \int_{\mathds{R}} | \mathds{1}_{(-\infty, \lambda)}(t) - f_{m}(t) |^{2} \mathrm{d}\bbnu(t) \right)^{1/2} \\
			& \quad + \left( \int_{\mathds{R}} | \mathds{1}_{(-\infty, \lambda)}(s) - f_{m}(s) |^{2} \mathrm{d} \tilde{\bbnu}(s) \right)^{1/2} \\
			& \quad + \left\| f_{m}(T) - p_{m, k}(T) \right\|_{\mathrm{op}} \\
			& \quad + \left\| f_{m}(T+S) - p_{m, k}(T+S) \right\|_{\mathrm{op}}
		\end{align*}
		for all $ m \in \mathds{N} $ and all $ k \in \mathds{N} $. 
		First, we send $ k \rightarrow \infty $ and then we take the limit $ m \rightarrow \infty $. 
		As $ m \rightarrow \infty $, the sequence $ | \mathds{1}_{(-\infty, \lambda)} - f_{m} |^{2} $ converges to zero pointwise almost everywhere with respect 
		to both $ \bbnu $ and $ \tilde{\bbnu} $. 
		Now (\ref{Approximiere stetige Funktion durch Polynome}) and the dominated convergence theorem imply
		\begin{equation*}
			\lim_{n \rightarrow \infty} \left\| \left( E_{(-\infty, \lambda)}(T+S) - E_{(-\infty, \lambda)}(T) \right) x_{n} \right\| = 0.
		\end{equation*}
		Recall that $ (x_{n})_{n \in \mathds{N}} $ is an orthonormal sequence.
		Thus, an application of Weyl's criterion (see, e.\,g., \cite[Proposition 8.11]{Schmuedgen}) concludes the proof.
	\end{proof}

	\begin{remark}
		If $ T $ is unbounded, then the spectrum of $ T $ is unbounded, so that the proof of Proposition \ref{Null im Spektrum bis auf abz} does not work. 
		For instance, we used the compactness of the spectrum in order to uniformly approximate $ f_{m} $ by polynomials.
		
		Moreover, it is unclear whether an orthonormal sequence $ (x_{n})_{n \in \mathds{N}} $ as in the proof of Proposition \ref{Null im Spektrum bis auf abz} 
		can be found in the domain of $ T $.
	\end{remark}
	
	\subsection{The case when the range of $ S $ is infinite dimensional}
	
	In this subsection, we suppose that $ S $ is compact and the range of $ S $ is infinite dimensional. The following lemma is easily shown.
	\begin{lemma} \label{Compact Operator Ergebnis}
		Let $ \lambda \in \mathds{R} \setminus \sigma_{\mathrm{ess}}(T) $. Then $ D(\lambda) $ is compact. 
	\end{lemma}	
	Furthermore, Proposition \ref{Null im Spektrum bis auf abz} still holds when $ S $ is compact and the range of $ S $ is infinite dimensional. 
	To see this, we need to modify two steps of the proof of Proposition \ref{Null im Spektrum bis auf abz}. 
	
	Let us write $ S = \sum_{j=1}^{\infty} \alpha_{j} \langle \cdot, \varphi_{j} \rangle \varphi_{j} $, 
	where $ \varphi_{1}, \varphi_{2}, ... $ is an orthonormal system in $ \mathfrak{H} $ and $ \alpha_{1}, \alpha_{2}, ... $ are nonzero real numbers.
	
	(1) In contrast to the proof of Proposition \ref{Null im Spektrum bis auf abz} above, we choose an orthonormal sequence $ x_{1}, x_{2}, ... $ in $ \mathfrak{H} $ as follows:
	\begin{align*}
		&x_{1} ~ \bot ~ \varphi_{1}, ~ x_{2} ~ \bot ~ \{ x_{1}, \varphi_{1}, \varphi_{2}, T \varphi_{1}, T \varphi_{2} \}, ~ ..., \\
		&x_{n} ~ \bot ~ \{ x_{1}, ..., x_{n-1}, \varphi_{k}, T \varphi_{k}, ..., T^{n-1} \varphi_{k} : k = 1, ..., n \}, ~ ...
	\end{align*}
	By construction, one has
	\begin{equation*}
		p(T + F_{\ell}) x_{n} = p(T) x_{n} \quad \text{ for all } n > \max (\ell, ~ \text{degree of } p),
	\end{equation*}
	where $ p $ is a polynomial, $ \ell \in \mathds{N} $, and $ F_{\ell} := \sum_{j=1}^{\ell} \alpha_{j} \langle \cdot, \varphi_{j} \rangle \varphi_{j} $.
	
	(2) We continue as in the proof of Proposition \ref{Null im Spektrum bis auf abz} above and estimate as follows:
	\begin{align*}
			&\limsup_{n \rightarrow \infty} \left\| \left( E_{(-\infty, \lambda)}(T+S) - E_{(-\infty, \lambda)}(T) \right) x_{n} \right\| \\
			&\leq \left( \int_{\mathds{R}} | \mathds{1}_{(-\infty, \lambda)}(t) - f_{m}(t) |^{2} \mathrm{d}\bbnu(t) \right)^{1/2} \\
			& \quad + \left( \int_{\mathds{R}} | \mathds{1}_{(-\infty, \lambda)}(s) - f_{m}(s) |^{2} \mathrm{d} \tilde{\bbnu}(s) \right)^{1/2} \\
			& \quad + \left\| f_{m}(T) - p_{m, k}(T) \right\|_{\mathrm{op}} \\
			& \quad + \left\| f_{m}(T+S) - p_{m, k}(T+S) \right\|_{\mathrm{op}} \\
			& \quad + \left\| p_{m,k}(T+S) - p_{m, k}(T + F_{\ell}) \right\|_{\mathrm{op}} 
	\end{align*}
	for all $ k, \ell, m \in \mathds{N} $, where $ \| \cdot \|_{\mathrm{op}} $ denotes the operator norm. 
	It is well known that the operators $ F_{\ell} $ uniformly approximate the operator $ S $ as $ \ell $ tends to infinity. 
	Therefore, $ \left\| p_{m,k}(T+S) - p_{m, k}(T + F_{\ell}) \right\|_{\mathrm{op}} \rightarrow 0 $ as $ \ell \rightarrow \infty $. 
	
	Analogously to the proof of Proposition \ref{Null im Spektrum bis auf abz} above, it follows that
	\begin{equation*}
			\lim_{n \rightarrow \infty} \left\| \left( E_{(-\infty, \lambda)}(T+S) - E_{(-\infty, \lambda)}(T) \right) x_{n} \right\| = 0.
	\end{equation*}
	
	Hence, we have shown that zero belongs to the essential spectrum of $ D(\lambda) $ for all but at most countably many $ \lambda \in \mathds{R} $. 
	
	\subsection{Proof of Theorem \ref{Main theorem III}}
		Taken together, Lemma \ref{Trace Class Ergebnis} and Proposition \ref{Null im Spektrum bis auf abz} show that Theorem \ref{Main theorem III} holds whenever the range of $ S $ is 
		finite dimensional.
		
		In the preceding subsection, we have shown that Theorem \ref{Main theorem III} also holds when $ S $ is compact and the range of $ S $ is infinite dimensional.
		
		Now the proof is complete. \qed
	\subsection{The smooth situation}

	In order to apply a result of Pushnitski \cite{Pushnitski_I} to $ D(\lambda) $, we check the corresponding assumptions stated in \cite[p.\,228]{Pushnitski_I}.
	
	First, define the compact self-adjoint operator $ G := |S|^{\frac{1}{2}} : \mathfrak{H} \rightarrow \mathfrak{H} $ 
	and the bounded self-adjoint operator $ S_{0} := \mathrm{sign}(S) : \mathfrak{H} \rightarrow \mathfrak{H} $. 
	Obviously, one has $ S = G^{\ast} S_{0} G $. Define the operator-valued functions $ h_{0} $ and $ h $ on $ \mathds{R} $ by
	\begin{align*}
		h_{0}(\lambda) = G E_{(-\infty, \lambda)}(T) G^{\ast}, \quad h(\lambda) = G E_{(-\infty, \lambda)}(T+S) G^{\ast}, \quad \lambda \in \mathds{R}.
	\end{align*}
	In order to fulfill \cite[Hypothesis 1.1]{Pushnitski_I}, we need the following assumptions.
	\begin{hypothesis}
		Suppose that there exists an open interval $ \delta $ contained in the absolutely continuous spectrum of $ T $. 
		Next, we assume that the derivatives
		\begin{align*}
			\dot{h}_{0}(\lambda) = \frac{\mathrm{d}}{\mathrm{d} \lambda} h_{0}(\lambda) \quad \text{and} \quad \dot{h}(\lambda) = \frac{\mathrm{d}}{\mathrm{d} \lambda} h(\lambda)
		\end{align*}
		exist in operator norm for all $ \lambda \in \delta $, and that the maps $ \delta \ni \lambda \mapsto \dot{h}_{0}(\lambda) $ 
		and \linebreak $ \delta \ni \lambda \mapsto \dot{h}(\lambda) $ are H\"older continuous (with some positive exponent) in the operator norm.
	\end{hypothesis}

	Now \cite[Theorem 1.1]{Pushnitski_I} yields that for all $ \lambda \in \delta $, there exists a nonnegative real number $ a $ such that
	\begin{equation*}
		\sigma_{\mathrm{ess}} \big( D(\lambda) \big) = \left[ -a, a \right].
	\end{equation*}
	The number $ a $ depends on $ \lambda $ and can be expressed in terms of the scattering matrix for the pair $ T $, $ T + S $, see \cite[Formula (1.3)]{Pushnitski_I}.
	\begin{example}	\label{Kreinsches Beispiel zum Zweiten}
		Again, consider Kre{\u\i}n's example \cite[pp.\,622--624]{Krein}. That is, $ \mathfrak{H} = L^{2}(0,\infty) $, the initial operator $ T = A_{0} $ is 
		the integral operator from Example \ref{via Krein}, and $ S = \langle \cdot, \varphi \rangle \varphi $ with $ \varphi(x) = \mathrm{e}^{-x} $. 
		Put $ \delta = (0,1) $. ~Then Pushnitski has shown in \cite[Subsection 1.3]{Pushnitski_I} that, by \cite[Theorem 1.1]{Pushnitski_I}, one has 
		$ \sigma_{\mathrm{ess}} \big( D(\lambda) \big) = [-1, 1] $ for all $ 0 < \lambda < 1 $.
		
		In particular, the operator $ D(\lambda) $ fulfills condition (C2) in Theorem \ref{Charakterisierung modifiziert} for all $ 0 < \lambda < 1 $.
	\end{example}
		
	\section{Proof of the main results}	\label{Beweis des Hauptresultates}
	This section is devoted to the proof of Theorem \ref{new Main result} and Theorem \ref{new Main result II}. First, we need to show two auxiliary results.
	
	\subsection{Two auxiliary results}	\label{Subsection with two auxiliary results}
	\begin{proposition}	\label{Main result}
		Let $ T $ and $ S = \langle \cdot, \varphi \rangle \varphi $ be a bounded self-adjoint operator and a self-adjoint operator of rank one acting on $ \mathfrak{H} $, respectively.
		\begin{enumerate}
			\item The operator $ D(\lambda) $ is unitarily equivalent to  a self-adjoint Hankel operator of finite rank for all 
						$ \lambda $ in $ \mathds{R} \setminus \left[ \min \sigma_{\mathrm{ess}}(T), \max \sigma_{\mathrm{ess}}(T) \right] $.
			\item \label{The second item of the main theorem}
						Suppose that $ \varphi $ is cyclic for $ T $. Then $ D(\lambda) $ is unitarily equivalent to a bounded self-adjoint Hankel operator for all $ \lambda $ in 
						$ \mathds{R} \setminus \sigma_{\mathrm{ess}}(T) $ and for all but at most countably many $ \lambda $ in $ \sigma_{\mathrm{ess}}(T) $.
		\end{enumerate}
	\end{proposition}
	\begin{proof}	
		(1) follows easily from Lemma \ref{Lemma zur Bedingung C3} and Theorem \ref{Charakterisierung modifiziert}, because
		\begin{align*}
			D(\lambda) &= E_{(-\infty, \lambda)}(T+S) - E_{(-\infty, \lambda)}(T) \\
								 &= E_{[ \lambda, \infty )}(T) - E_{[ \lambda, \infty )}(T+S)
		\end{align*}
		is a finite rank operator for all $ \lambda $ in $ \left( -\infty, \min \sigma_{\mathrm{ess}}(T) \right) \cup \left( \max \sigma_{\mathrm{ess}}(T), \infty \right) $.
	
		(2) is a direct consequence of Lemma \ref{Lemma zur Bedingung C3}, Theorem \ref{Main theorem II}, Theorem \ref{Main theorem III}, and Theorem \ref{Charakterisierung modifiziert}. 
		
		This concludes the proof.
	\end{proof}
	
	\begin{lemma}	\label{Vorbereitung zu new Main result II}
		The statements of Theorem \ref{new Main result II} hold if $ T $ is additionally assumed to be bounded.
	\end{lemma}
	\begin{proof} 
		Let $ \lambda \in \mathds{R} $. It follows from Halmos' decomposition (see \cite{Halmos}) of $ \mathfrak{H} $ with respect to the orthogonal projections 
		$ E_{(-\infty, \lambda)}(T+S) $ and $ E_{(-\infty, \lambda)}(T) $ that we obtain the following orthogonal decomposition of $ \mathfrak{H} $ with respect to $ D(\lambda) $:
		\begin{align*}
			\mathfrak{H} 
			= \Big( \mathrm{Ker} ~ D(\lambda) \Big)
				\oplus \Big( \mathrm{Ran} ~ E_{\{ 1 \}} \big( D(\lambda) \big) \Big)
				\oplus \Big( \mathrm{Ran} ~ E_{\{ -1 \}} \big( D(\lambda) \big) \Big)
				\oplus \mathfrak{H}_{g}^{(\lambda)}.
		\end{align*}
		Here $ \mathfrak{H}_{g}^{(\lambda)} $ is the orthogonal complement of 
		\begin{equation*}
			\widetilde{\mathfrak{H}}^{(\lambda)} 
			:= \Big( \mathrm{Ker} ~ D(\lambda) \Big)
					\oplus \Big( \mathrm{Ran} ~ E_{\{ 1 \}} \big( D(\lambda) \big) \Big)
					\oplus \Big( \mathrm{Ran} ~ E_{\{ -1 \}} \big( D(\lambda) \big) \Big)
		\end{equation*}
		in $ \mathfrak{H} $. Clearly, $ \mathfrak{H}_{g}^{(\lambda)} $ is reducing for the operator $ D(\lambda) $. 
		It follows from Lemma \ref{Satz bei Chandler Davis} that $ \left. D(\lambda) \right|_{\mathfrak{H}_{g}^{(\lambda)}} $ 
		is unitarily equivalent to $ - \! \left. D(\lambda) \right|_{\mathfrak{H}_{g}^{(\lambda)}} $.
			
		It is elementary to show that there exists a compact self-adjoint block diagonal operator $ \widetilde{K}_{\lambda} \oplus 0 $ on 
		$ \widetilde{\mathfrak{H}}^{(\lambda)} \oplus \mathfrak{H}_{g}^{(\lambda)} $ with the following properties:
		\begin{itemize}
			\item $ \widetilde{K}_{\lambda} \oplus 0 $ fulfills assertion (2) in Theorem \ref{new Main result II}.
			\item the range of $ \widetilde{K}_{\lambda} \oplus 0 $ is infinite dimensional if and only if one of the closed subspaces 
						$ \mathrm{Ran} ~ E_{\{ 1 \}} \big( D(\lambda) \big), ~ \mathrm{Ran} ~ E_{\{ -1 \}} \big( D(\lambda) \big) $ is finite dimensional and the other one is infinite dimensional.
			\item the kernel of $ D(\lambda) - \big( \widetilde{K}_{\lambda} \oplus 0 \big) $ is either trivial or infinite dimensional.
			\item if $ \widetilde{\mathfrak{H}}^{(\lambda)} \neq \{ 0 \} $, then the spectrum of 
						$ \left. D(\lambda) \right|_{\widetilde{\mathfrak{H}}^{(\lambda)}} - \widetilde{K}_{\lambda} $ is contained in the interval $ [-1,1] $ and 	
						consists only of eigenvalues. Moreover, the dimensions of 
						$ \mathrm{Ran} ~ E_{\{ t \}} \big( \! \! \left. D(\lambda) \right|_{\widetilde{\mathfrak{H}}^{(\lambda)}} - \widetilde{K}_{\lambda} \big) $ and 
						$	\mathrm{Ran} ~ E_{\{ -t \}} \big( \! \! \left. D(\lambda) \right|_{\widetilde{\mathfrak{H}}^{(\lambda)}} - \widetilde{K}_{\lambda} \big) $ 
						differ by at most one, for all $ 0 < t \leq 1 $.
		\end{itemize}
		The block diagonal operator $ \widetilde{K}_{\lambda} \oplus 0 $ serves as a correction term for $ D(\lambda) $. 
		In particular, no correction term is needed if $ \widetilde{\mathfrak{H}}^{(\lambda)} = \{ 0 \} $.
		
		Theorem \ref{Main theorem III} and the invariance of the essential spectrum under compact perturbations imply that zero belongs to the essential spectrum of 
		$ D(\lambda) - \big( \widetilde{K}_{\lambda} \oplus 0 \big) $ for all $ \lambda $ in $ \mathds{R} $ except for at most countably many $ \lambda $ in $ \sigma_{\mathrm{ess}}(T) $.
		
		Therefore, an application of \cite[Theorem 1]{Megretskii_et_al} yields that $ D(\lambda) - \big( \widetilde{K}_{\lambda} \oplus 0 \big) $ is unitarily equivalent to a bounded 
		self-adjoint Hankel operator $ \Gamma_{\lambda} $ on $ \ell^{2}(\mathds{N}_{0}) $ for all $ \lambda $ in $ \mathds{R} $ 
		except for at most countably many $ \lambda $ in $ \sigma_{\mathrm{ess}}(T) $.
		
		Thus, by the properties of $ \widetilde{K}_{\lambda} \oplus 0 $ listed above, the claim follows.
	\end{proof}
	
	\begin{remark}
		If we consider $ E_{(-\infty, \lambda]}(T) - E_{(-\infty, \lambda]}(T + S) $, the difference of the spectral projections associated with 
		the closed interval $ (-\infty, \lambda] $ instead of the open interval $ (-\infty, \lambda) $, 
		then all assertions in Lemma \ref{Lemma zur Bedingung C3}, Proposition \ref{Einige hinreichende Bedingungen}, Theorem \ref{Main theorem II}, Theorem \ref{Main theorem III},
		Proposition \ref{Main result}, and Lemma \ref{Vorbereitung zu new Main result II} remain true. 
		All proofs can easily be modified.
	\end{remark}

	\subsection{The case when $ T $ is semibounded}	\label{Zum halbbeschraenkten Fall}
	In this subsection, which is based on Kre{\u\i}n's approach in \cite[pp.\,622--623]{Krein}, we deal with the case when the self-adjoint operator $ T $ 
	is semibounded \emph{but not bounded}. 
	As before, we write 
	\begin{equation*}
		D(\lambda) = E_{(-\infty, \lambda)}(T+S) - E_{(-\infty, \lambda)}(T)  
	\end{equation*}
	if $ S $ is a compact self-adjoint operator and $ \lambda \in \mathds{R} $. 
	
	First, consider the case when $ T $ is bounded from below. Choose $ c \in \mathds{R} $ such that
	\begin{align}	\label{Der nach unten beschraenkte Fall}
		T + cI \geq 0	\quad	\text{and} \quad	T + S + cI \geq 0.
	\end{align}
	It suffices to consider $ D(\lambda) $ for $ \lambda \geq -c $. Compute
	\begin{align*}
		D(\lambda) 
		&= E_{[\lambda, \infty)}(T) - E_{[\lambda, \infty)}(T+S) \\
		&= E_{(-\infty, \mu ]} \big( (T + (1+c)I)^{-1} \big) - E_{(-\infty, \mu ]} \big( (T + S + (1+c)I)^{-1} \big),
	\end{align*}
	where $ \mu = \frac{1}{\lambda + 1 + c} $. By the second resolvent equation, one has
	\begin{equation*}
		\big( T + S + (1+c)I \big)^{-1} = \big( T + (1+c)I \big)^{-1} - \big( T + S + (1+c)I \big)^{-1} S \big( T + (1+c)I \big)^{-1}.
	\end{equation*}
	The operator 
	\begin{equation}	\label{Festlegung von S strich}
		S^{\prime} := - \big( T + S + (1+c)I \big)^{-1} S \big( T + (1+c)I \big)^{-1}
	\end{equation}	
	is compact and self-adjoint. One can easily show that $ \mathrm{rank} ~ S^{\prime} = \mathrm{rank} ~ S $.
	
	In particular, if $ S = \langle \cdot, \varphi \rangle \varphi $ has rank one and $ \varphi^{\prime} := \frac{(T + (1+c)I)^{-1} \varphi}{\| (T + (1+c)I)^{-1} \varphi \|} $, 
	then there exists a number $ \alpha^{\prime} \in \mathds{R} $ such that $ S^{\prime} = \alpha^{\prime} \langle \cdot, \varphi^{\prime} \rangle \varphi^{\prime} $.
	
	We have shown:
	\begin{lemma}	\label{The bounded case again}
		Let $ T $ be a self-adjoint operator which is bounded from below but not bounded, let $ S $ be a compact self-adjoint operator, 
		and let $ c $ be such that (\ref{Der nach unten beschraenkte Fall}) holds. 
		~Then $ D(\lambda) = 0 $ for all $ \lambda < -c $ and 
		\begin{align*}
			D(\lambda) 
			= E_{(-\infty, \mu]} \big( T^{\prime} \big) - E_{(-\infty, \mu]} \big( T^{\prime} + S^{\prime} \big)
			\quad \text{for all } \lambda \geq -c.
		\end{align*}
		Here $ \mu = \frac{1}{\lambda + 1 + c} $\,, $ T^{\prime} = \big( T + (1+c)I \big)^{-1} $, and $ S^{\prime} $ is defined as in (\ref{Festlegung von S strich}).
	\end{lemma}
	The case when $ T $ is bounded from below can now be pulled back to the bounded case, see the remark in Subsection \ref{Subsection with two auxiliary results} above.
	
	\begin{proposition}	\label{Hauptergebnis im nach unten beschraenkten Fall}
		Suppose that $ S = \langle \cdot, \varphi \rangle \varphi $ is a self-adjoint operator of rank one and 
		that $ T $ is a self-adjoint operator which is bounded from below but not bounded. 
		Assume further that the spectrum of $ T $ is not purely discrete and that the vector 
		$ \varphi $ is cyclic for $ T $. ~Then the kernel of $ D(\lambda) $ is trivial for all $ \lambda > \min \sigma_{\mathrm{ess}}(T) $.
		
		Furthermore, $ D(\lambda) $ is unitarily equivalent to a bounded self-adjoint Hankel operator for all $ \lambda $ in $ \mathds{R} \setminus \sigma_{\mathrm{ess}}(T) $ and for all 
		but at most countably many $ \lambda $ in $ \sigma_{\mathrm{ess}}(T) $.
	\end{proposition}
	\begin{proof}
		Let $ c $ be such that (\ref{Der nach unten beschraenkte Fall}) holds. 
	
		It is easy to show that $ (T+(1+c)I)^{-1} \varphi $ is cyclic for $ (T+(1+c)I)^{-1} $ if $ \varphi $ is cyclic for $ T $. 
	
		Furthermore, it is easy to show that the function $ x \mapsto \frac{1}{x+1+c} $ is one-to-one from $ \sigma_{\mathrm{ess}}(T) $ 
		onto $ \sigma_{\mathrm{ess}} \big( (T + (1+c)I)^{-1} \big) \setminus \{ 0 \} $.
	
		One has that $ \min \sigma_{\mathrm{ess}} \big( (T + (1+c)I)^{-1} \big) = 0 $ and, since the spectrum of $ T $ is not purely discrete, 
		$ \max \sigma_{\mathrm{ess}} \big( (T + (1+c)I)^{-1} \big) = \frac{1}{\lambda_{0}+1+c} $, where
		$ \lambda_{0} := \min \sigma_{\mathrm{ess}}(T) $. Therefore, $ \mu = \frac{1}{\lambda+1+c} $ belongs to the open interval $ \Big( 0, \frac{1}{\lambda_{0}+1+c} \Big) $ 
		if and only if $ \lambda > \lambda_{0} $.

		In view of Lemma \ref{The bounded case again} and the remark in Subsection \ref{Subsection with two auxiliary results} above, the claims follow.
	\end{proof}
	Moreover, standard computations show: 
	\begin{corollary}	\label{Same list of sufficient conditions}
		Suppose that $ S $ is a self-adjoint operator of finite rank $ N \in \mathds{N} $ and that $ T $ is a self-adjoint operator which is bounded from below but not bounded. 
		We obtain the same list of sufficient conditions for $ D(\lambda) $ to be unitarily equivalent to a bounded self-adjoint block-Hankel operator of order $ N $ 
		with infinite dimensional kernel for all $ \lambda \in \mathds{R} $ as in Proposition \ref{Einige hinreichende Bedingungen} above.		
	\end{corollary}
	
	\begin{proof}
		Let $ X = T $ or $ X = T+S $ and let $ c $ be such that (\ref{Der nach unten beschraenkte Fall}) holds. One has:
		\begin{itemize} [leftmargin=2.5em]
			\item The real number $ \lambda $ is an eigenvalue of $ X $ with multiplicity $ k \in \mathds{N} \cup \{ \infty \} $ if and only if $ \frac{1}{\lambda+1+c} $ is an eigenvalue of 
						$ (X+(1+c)I)^{-1} $ with the same multiplicity $ k $. 
			\item The spectrum of the restricted operator $ \left. X \right|_{\mathfrak{E}^{\bot}} $ has multiplicity at least $ N+1 $ if and only if 
						the spectrum of the restricted operator $ \left. (X+(1+c)I)^{-1} \right|_{\mathfrak{E}^{\bot}} $ has multiplicity at least $ N+1 $, 
						where $ \mathfrak{E} := \left\{ x \in \mathfrak{H} : x \text{ is an eigenvector of } X \right\} $.
		\end{itemize}
		
		In view of Lemma \ref{The bounded case again} and the remark in Subsection \ref{Subsection with two auxiliary results} above, the claim follows.
	\end{proof}
	In Proposition \ref{Hauptergebnis im nach unten beschraenkten Fall}, we assumed that the spectrum of $ T $ is not purely discrete. 
	Now consider the case when $ T $ has a purely discrete spectrum. By the invariance of the essential spectrum under compact perturbations, it is clear that
	the operator $ T + S $ also has a purely discrete spectrum, for all compact self-adjoint operators $ S $. 
	Moreover, since $ T $ is bounded from below, we know that $ T + S $ is bounded from below as well. 
	Therefore, the range of $ D(\lambda) $ is finite dimensional for all $ \lambda \in \mathds{R} $, 
	and in particular conditions (C1) and (C2) in Theorem \ref{Charakterisierung modifiziert allgemeinere Version} are fulfilled.
	
	Combining this with Lemma \ref{Lemma zur Bedingung C3} and the remark in Subsection \ref{Subsection with two auxiliary results} above, we have shown:
	\begin{proposition}	\label{Keine Ausnahmepunkte im Fall von purely discrete spectrum}
		Suppose that $ S $ is a self-adjoint operator of finite rank $ N \in \mathds{N} $ and that $ T $ is a bounded from below self-adjoint operator with a purely discrete spectrum. 
		Then $ D(\lambda) $ is unitarily equivalent to a finite rank self-adjoint block-Hankel operator of order $ N $ for \emph{all} $ \lambda \in \mathds{R} $.
	\end{proposition}
	This proposition supports the idea that there is a structural correlation between the operator $ D(\lambda) $ and block-Hankel operators.
	
	Now, consider the case when $ T $ is bounded from above. Choose $ c \in \mathds{R} $ such that
	\begin{align*}
		T - cI \leq 0	\quad	\text{and} \quad	T + S - cI \leq 0.
	\end{align*}
	It suffices to consider $ D(\lambda) $ for $ \lambda \leq c $. Compute
	\begin{align*}
		D(\lambda)
		&= E_{(\mu, \infty)} \big( (T+S-(1+c)I )^{-1} \big) - E_{(\mu, \infty)} \big( (T-(1+c)I)^{-1} \big) \\
		&= E_{(-\infty, \mu]} \big( (T-(1+c)I)^{-1} \big) - E_{(-\infty, \mu]} \big( (T+S-(1+c)I )^{-1} \big),
	\end{align*}
	where $ \mu = \frac{1}{\lambda - (1 + c)} $. By the second resolvent equation, one has
	\begin{equation*}
		\big( T + S - (1+c)I \big)^{-1} = \big( T - (1+c)I \big)^{-1} - \big( T + S - (1+c)I \big)^{-1} S \big( T - (1+c)I \big)^{-1}.
	\end{equation*}
	The operator $ S^{\prime \prime} := - \big( T + S - (1+c)I \big)^{-1} S \big( T - (1+c)I \big)^{-1} $ 
	is compact and self-adjoint with $ \mathrm{rank} ~ S^{\prime \prime} = \mathrm{rank} ~ S $. 
	
	In particular, if $ S = \langle \cdot, \varphi \rangle \varphi $ has rank one and $ \varphi^{\prime \prime} := \frac{(T - (1+c)I)^{-1} \varphi}{\| (T - (1+c)I)^{-1} \varphi \|} $, 
	then there exists a number $ \alpha^{\prime \prime} \in \mathds{R} $ 
	such that $ S^{\prime \prime} = \alpha^{\prime \prime} \langle \cdot, \varphi^{\prime \prime} \rangle \varphi^{\prime \prime} $.
	
	Now proceed analogously to the case when $ T $ is bounded from below. 
	
	It follows that Proposition \ref{Hauptergebnis im nach unten beschraenkten Fall} holds 
	in the case when $ T $ is bounded from above but not bounded if we replace $ \lambda > \min \sigma_{\mathrm{ess}}(T) $ by $ \lambda < \max \sigma_{\mathrm{ess}}(T) $.
	
	Furthermore, Corollary \ref{Same list of sufficient conditions} still holds if $ T $ is bounded from above but not \linebreak bounded.
			
	Obviously, Proposition \ref{Keine Ausnahmepunkte im Fall von purely discrete spectrum} holds in the case when $ T $ is bounded from above.	\label{Ergaenzung Proposition}
	
	\subsection{Proof of Theorem \ref{new Main result} and Theorem \ref{new Main result II}}
	Let us first complete the proof of Theorem \ref{new Main result}.
	\begin{proof}[Proof of Theorem \ref{new Main result}] 
		If the operator $ T $ is bounded, then the statement of Theorem \ref{new Main result} follows from Proposition \ref{Main result} 
		and the discussion of Case 1 -- Case 3 in Subsection \ref{Reduktion zum zyklischen Fall} above.

		Now suppose that $ T $ is bounded from below but not bounded and let $ c $ be such that (\ref{Der nach unten beschraenkte Fall}) holds. 
		First, assume that the spectrum of $ T $ is not purely discrete. 
		~If $ \varphi $ is cyclic for $ T $, then the claim follows from Proposition \ref{Hauptergebnis im nach unten beschraenkten Fall}.	
		~In the case when $ \varphi $ is not cyclic for $ T $, we consider the bounded operator $ T^{\prime} $ and the rank one operator $ S^{\prime} $ 
		defined as in Lemma \ref{The bounded case again} above.
		As we have noted in the proof of Proposition \ref{Hauptergebnis im nach unten beschraenkten Fall}, 
		it is easy to show that the function $ x \mapsto \frac{1}{x+1+c} $ is one-to-one from $ \sigma_{\mathrm{ess}}(T) $ 
		onto $ \sigma_{\mathrm{ess}} \big( T^{\prime} \big) \setminus \{ 0 \} $. 
		Now the statement of Theorem \ref{new Main result} follows from the remark in Subsection \ref{Subsection with two auxiliary results}, Proposition \ref{Main result},
		and the discussion of Case 1 -- Case 3 in Subsection \ref{Reduktion zum zyklischen Fall} above.
		
		If $ T $ has a purely discrete spectrum, then Proposition \ref{Keine Ausnahmepunkte im Fall von purely discrete spectrum} shows that 
		$ D(\lambda) $ is unitarily equivalent to a finite rank self-adjoint Hankel operator for all $ \lambda \in \mathds{R} $.
		
		If $ T $ is bounded from above but not bounded, then the proof runs analogously.
		
		This finishes the proof.
	\end{proof}
	 Now let us prove Theorem \ref{new Main result II}.
	\begin{proof}[Proof of Theorem \ref{new Main result II}] 
		In view of Lemma \ref{Vorbereitung zu new Main result II}, it suffices to consider the case when $ T $ is semibounded but not bounded.
	
		First, let $ T $ be bounded from below but not bounded and let $ c $ be such that (\ref{Der nach unten beschraenkte Fall}) holds. 
		Again, recall that the function $ x \mapsto \frac{1}{x+1+c} $ is one-to-one from $ \sigma_{\mathrm{ess}}(T) $ 
		onto $ \sigma_{\mathrm{ess}} \big( (T + (1+c)I)^{-1} \big) \setminus \{ 0 \} $. 		
		Now the statements of Theorem \ref{new Main result II} follow from Lemma \ref{The bounded case again}, the remark in Subsection \ref{Subsection with two auxiliary results} above, 
		and Lemma \ref{Vorbereitung zu new Main result II}.
	
		If $ T $ is bounded from above but not bounded, then the proof runs analogously.
	\end{proof}
	
	\section{Some examples}	\label{Beispiele und Anwendungen}
	
	In this section, we apply the above theory in the context of operators that are of particular interest in various fields of (applied) mathematics, 
	such as Schr\"odinger operators.
	
	In any of the following examples, the operator $ D(\lambda) $ is unitarily equivalent to a bounded self-adjoint (block-) Hankel operator for all $ \lambda $ in $ \mathds{R} $.  
	
	First, we consider the case when $ T $ has a purely discrete spectrum.	\newpage
	
	\begin{example}
		Let $ \mathfrak{H} = L^{2}(\mathds{R}^{n}) $ and suppose that $ V \geq 0 $ is in $ L_{\mathrm{loc}}^{1}(\mathds{R}^{n}) $ such that Lebesgue measure 
		of $ \{ x \in \mathds{R}^{n} : 0 \leq V(x) < M \} $ is finite for all $ M > 0 $. 
		Then the self-adjoint Schr\"odinger operator $ T \geq 0 $ defined by the form sum of $ - \Delta $ and $ V $ has a purely discrete spectrum, 
		see \cite[Example 4.1]{Wang-Wu}; see also \cite[Theorem 1]{Simon_B_II}. 		
		Therefore, if $ S $ is any self-adjoint operator of finite rank $ N $, then Proposition \ref{Keine Ausnahmepunkte im Fall von purely discrete spectrum} implies that 
		$ D(\lambda) $ is unitarily equivalent to a finite rank self-adjoint block-Hankel operator of order $ N $ for all $ \lambda \in \mathds{R} $.
	\end{example}
	
	Next, consider the case when $ S = \langle \cdot, \varphi \rangle \varphi $ is of rank one and $ \varphi $ is cyclic for $ T $. 

	\begin{example}	\label{Rekonstruktion Kreinsches Beispiel}
		Once again, consider Kre{\u\i}n's example \cite[pp.\,622--624]{Krein}. 
		
		The operators $ T = A_{0} $ and $ T + \langle \cdot, \varphi \rangle \varphi  = A_{1} $, 
		where $ \varphi(x) = \mathrm{e}^{-x} $,
		from Example \ref{via Krein} both have a simple purely absolutely continuous spectrum filling in the interval $ [0,1] $. 
		Therefore, $ D(\lambda) $ is the zero operator for all $ \lambda \in \mathds{R} \setminus (0,1) $.
		
		\emph{(}$ \ast $\emph{)} \hspace{2ex} The function $ \varphi $ is cyclic for $ T $. \newline	
		Hence, Theorem \ref{Main theorem II} implies that the kernel of $ D(\lambda) $ is trivial for all $ 0 < \lambda < 1 $. 
				
		Furthermore, an application of Proposition \ref{Main result} yields that $ D(\lambda) $ is unitarily equivalent 
		to a bounded self-adjoint Hankel operator for all $ \lambda $ in $ \mathds{R} $ except for at most countably many $ \lambda $ in $ [0,1] $.
	
		Note that, in this example, explicit computations show that there are no exceptional points (see \cite{Krein}).
	\end{example}
	\begin{proof}[Proof of \emph{(}$ \ast $\emph{)}]
		Let $ k $ be in $ \mathds{N}_{0} $. Define the  \(k\)th Laguerre polynomial $ L_{k} $ on $ (0,\infty) $ 
		by $ L_{k}(x) := \frac{\mathrm{e}^{x}}{k!} ~ \frac{\mathrm{d}^{k}}{\mathrm{d}x^{k}} (x^{k} \mathrm{e}^{-x}) $.
		Furthermore, define $ \psi_{k} $ on $ (0,\infty) $ by $ \psi_{k}(x) := x^{k} \mathrm{e}^{-x} $. A straightforward computation shows that
		\begin{align*}
			\big( A_{0} \psi_{k} \big) (x) 
			= \frac{1}{2} \, \mathrm{e}^{-x} \bigg\{ \frac{x^{k+1}}{k+1} + \frac{1}{2^{k+1}} \sum_{\ell=0}^{k-1} (2x)^{k-\ell} ~ \frac{k!}{(k-\ell)!} \bigg\}.
		\end{align*}
		By induction on $ n \in \mathds{N}_{0} $, it easily follows that $ p \cdot \varphi $ belongs 
		to the linear span of $ A_{0}^{\ell} \varphi $, $ \ell \in \mathds{N}_{0} $, $ \ell \leq n $, for all polynomials $ p $ of degree $ \leq n $.
		
		In particular, the functions $ \phi_{j} $ defined on $ (0, \infty) $ by $ \phi_{j}(x) := \sqrt{2} ~ L_{j}(2x) \mathrm{e}^{-x} $ are elements 
		of $ \mathrm{span} \big\{ A_{0}^{\ell} \varphi : \ell \in \mathds{N}_{0}, \ell \leq n \big\} $ for all $ j \in \mathds{N}_{0} $ with $ j \leq n $.
		
		Since $ (\phi_{j})_{j \in \mathds{N}_{0}} $ is an orthonormal basis of $ L^{2}(0,\infty) $, it follows that $ \varphi $ is cyclic for $ T $.
	\end{proof}
	Example \ref{Rekonstruktion Kreinsches Beispiel} suggests the conjecture that Proposition \ref{Main result} (\ref{The second item of the main theorem}) can be 
	strengthened to hold up to a finite exceptional set.
		
	Last, we consider different examples where the multiplicity in the spectrum of $ T $ is such that we can apply Proposition \ref{Einige hinreichende Bedingungen}. 
	\begin{example}
		\begin{enumerate}
			\item Let $ T $ be an arbitrary orthogonal projection on $ \mathfrak{H} $, and let $ S $ be a self-adjoint operator of finite rank.  
						Then zero or one is an eigenvalue of $ T $ with infinite multiplicity, and we can apply Proposition \ref{Einige hinreichende Bedingungen}.
			\item Put $ \mathfrak{H} = L^{2}(0,\infty) $ and let $ T $ be the Carleman operator, i.\,e., the bounded Hankel operator such that
						\begin{align*}
							(Tg)(x) = \int_{0}^{\infty} \frac{g(y)}{x+y} \mathrm{d}y
						\end{align*}
						for all continuous functions $ g : (0,\infty) \rightarrow \mathds{C} $ with compact support. \newpage
						
						It is well known (see, e.\,g., \cite[Chapter 10, Theorem 2.3]{Peller}) that the Carleman operator has a purely absolutely continuous spectrum 
						of uniform multiplicity two filling in the interval $ [0, \uppi] $. 
						Therefore, if $ S $ is any self-adjoint operator of rank one, Proposition \ref{Einige hinreichende Bedingungen} can be applied.
		\end{enumerate}		
	\end{example}
	
	\subsection*{Jacobi operators}
	
	Consider a bounded self-adjoint Jacobi operator $ H $ acting on the Hilbert space $ \ell^{2}(\mathds{Z}) $ of complex square summable two-sided sequences. \linebreak
	More precisely, suppose that there exist bounded real-valued sequences $ a = (a_{n})_{n} $ and $ b = (b_{n})_{n} $ with $ a_{n} > 0 $ for all $ n \in \mathds{Z} $ such that
	\begin{align*}
		(H x)_{n} = a_{n} x_{n+1} + a_{n-1} x_{n-1} + b_{n} x_{n}, \quad n \in \mathds{Z},
	\end{align*}
	cf.\,\cite[Theorem 1.5 and Lemma 1.6]{Teschl}. The following result is well known.
	\begin{proposition}[see \cite{Teschl}, Lemma 3.6]
		Let $ H $ be a bounded self-adjoint Jacobi operator on $ \ell^{2}(\mathds{Z}) $.		
		Then the singular spectrum of $ H $ has spectral multiplicity one, and 
		the absolutely continuous spectrum of $ H $ has multiplicity at most two.
	\end{proposition}
	In the case where $ H $ has a simple spectrum, there exists a cyclic vector $ \varphi $ for $ H $, and we can apply Proposition \ref{Main result} to $ H $ with  
	the rank one perturbation $ S = \langle \cdot, \varphi \rangle \varphi $.
	
	Otherwise, $ H $ fulfills condition (\ref{Vielfachheit groesser eins im stetigen Spektrum}) with $ N = 1 $ in Proposition \ref{Einige hinreichende Bedingungen}. 
	Let us discuss some examples in the latter case with $ T = H $. Since $ S $ can be an arbitrary self-adjoint operator of rank one, we do not mention it in the following. 
	\begin{example}
		Consider the discrete Schr\"odinger operator $ H = H_{V} $ on $ \ell^{2}(\mathds{Z}) $ with bounded potential $ V : \mathds{Z} \rightarrow \mathds{R} $,
		\begin{align*}
			\big( H_{V}x \big)_{n} = x_{n+1} + x_{n-1} + V_{n} x_{n}, \quad n \in \mathds{Z}.
		\end{align*}
		If the spectrum of $ H_{V} $ contains only finitely many points outside of the interval $ [-2,2] $, then \cite[Theorem 2]{Damanik_et_al} implies that $ H_{V} $ has 
		a purely absolutely continuous spectrum of multiplicity two on $ [-2,2] $. 
		
		It is well known that the free Jacobi operator $ H_{0} $ with $ V = 0 $ has a purely absolutely continuous spectrum of multiplicity two filling in the interval $ [-2,2] $.
	\end{example}
	
	Let us consider the almost Mathieu operator $ H = H_{\kappa, \beta, \theta} : \ell^{2}(\mathds{Z}) \rightarrow \ell^{2}(\mathds{Z}) $ defined by
	\begin{align*}
		(H x)_{n} = x_{n+1} + x_{n-1} + 2 \kappa \cos \! \big( 2 \uppi (\theta + n \beta) \big) x_{n}, \quad n \in \mathds{Z},
	\end{align*}
	where $ \kappa \in \mathds{R} \setminus \{ 0 \} $ and $ \beta, \theta \in \mathds{R} $. 
	In fact, it suffices to consider $ \beta, \theta \in \mathds{R} / \mathds{Z} $. 
	
	The almost Mathieu operator plays an important role in physics, see, for instance, the review \cite{Last_Y} and the references therein. 
	
	Here, we are interested in cases where Proposition \ref{Einige hinreichende Bedingungen} can be applied to the \linebreak almost Mathieu operator 
	with an arbitrary self-adjoint rank one perturbation. \linebreak Sufficient conditions for this purpose are provided in the following lemma.
	\begin{lemma}	\label{Example almost Mathieu operator}
		\begin{enumerate}
			\item If $ \beta $ is rational, then for all $ \kappa $ and $ \theta $ the almost Mathieu \linebreak operator $ H_{\kappa, \beta, \theta} $ is periodic and has 
						a purely absolutely continuous spectrum of uniform multiplicity two.
			\item If $ \beta $ is irrational and $ | \kappa | < 1 $, then for all $ \theta $ the almost Mathieu operator $ H_{\kappa, \beta, \theta} $ 
						has a purely absolutely continuous spectrum of uniform multiplicity two.
		\end{enumerate}
	\end{lemma}
	
	\begin{proof}
		(1) If $ \beta $ is rational, then $ H_{\kappa, \beta, \theta} $ is a periodic Jacobi operator. 
		Hence, it is well known (see, e.\,g., \cite[p.\,122]{Teschl}) that the spectrum of $ H_{\kappa, \beta, \theta} $ is purely absolutely continuous. 
		According to \cite[Theorem 9.1]{Deift_Simon}, we know that the absolutely continuous spectrum of $ H_{\kappa, \beta, \theta} $ is uniformly of multiplicity two. This proves (1).
		
		(2) Suppose that $ \beta $ is irrational. Avila has shown (see \cite[Main Theorem]{Avila}) that the almost Mathieu operator $ H_{\kappa, \beta, \theta} $ has 
		a purely absolutely continuous spectrum if and only if $ | \kappa | < 1 $. 
		Again, \cite[Theorem 9.1]{Deift_Simon} implies that the absolutely continuous spectrum of $ H_{\kappa, \beta, \theta} $ is uniformly of multiplicity two. 
		This finishes the proof.
	\end{proof}
	Problems 4--6 of Simon's list \cite{Simon_B} are concerned with the almost Mathieu operator. 
	Avila's result \cite[Main Theorem]{Avila}, which we used in the above proof, is a solution for Problem 6 in \cite{Simon_B}.

  \section*{Acknowledgments}		
		The author would like to thank his Ph.D. supervisor, Vadim Kostrykin, for many fruitful discussions. 
		The author would like to thank Julian Gro{\ss}mann, Stephan Schmitz, and Albrecht Seelmann for reading the manuscript and Konstantin \linebreak A. Makarov for editorial advice. 
		Furthermore, the author would like to thank Alexander Pushnitski for his invitation to give a talk at the King's College Analysis Seminar and for helpful discussions 
		on differences of functions of operators.

\end{document}